\documentclass[11pt,reqno]{amsart}

\usepackage[T1]{fontenc}
\usepackage[utf8]{inputenc}
\usepackage{lmodern}
\usepackage{microtype}
\usepackage[margin=1in]{geometry}
\usepackage{amsmath,amssymb,mathtools,mathrsfs,esint}
\usepackage{enumitem}
\usepackage{xcolor}
\usepackage{hyperref}
\usepackage{aliascnt}
\usepackage[nameinlink,capitalise,noabbrev]{cleveref}

\hypersetup{
  colorlinks=true,
  linkcolor=blue!45!black,
  citecolor=magenta,
  urlcolor=blue!45!black,
  pdftitle={Uniform logarithmic Sobolev inequalities for the 2D Coulomb gas at the diffusive temperature scale},
  pdfauthor={Matthew Rosenzweig}
}

\allowdisplaybreaks
\setlist[itemize]{topsep=4pt,itemsep=2pt,parsep=1pt}
\setlist[enumerate]{topsep=4pt,itemsep=3pt,parsep=1pt}
\numberwithin{equation}{section}

\theoremstyle{plain}
\newtheorem{theorem}{Theorem}[section]
\newaliascnt{proposition}{theorem}
\newtheorem{proposition}[proposition]{Proposition}
\aliascntresetthe{proposition}
\newaliascnt{lemma}{theorem}
\newtheorem{lemma}[lemma]{Lemma}
\aliascntresetthe{lemma}
\newaliascnt{corollary}{theorem}
\newtheorem{corollary}[corollary]{Corollary}
\aliascntresetthe{corollary}

\theoremstyle{definition}
\newaliascnt{definition}{theorem}

\aliascntresetthe{definition}

\theoremstyle{remark}
\newaliascnt{remark}{theorem}
\newtheorem{remark}[remark]{Remark}
\aliascntresetthe{remark}

\newcommand{\C}{\mathbb C}
\newcommand{\R}{\mathbb R}
\newcommand{\T}{\mathbb T}

\newcommand{\E}{\mathbb E}
\newcommand{\g}{\mathsf g}
\newcommand{\s}{\mathsf{s}}
\renewcommand{\d}{\mathsf{d}}
\newcommand{\D}{\,\mathrm d}
\newcommand{\Ent}{\operatorname{Ent}}
\newcommand{\Var}{\operatorname{Var}}

\newcommand{\cE}{\mathcal E}
\newcommand{\cF}{\mathcal F}
\newcommand{\cS}{\mathscr S}
\newcommand{\sfS}{\mathsf S}
\newcommand{\1}{\mathbf 1}

\title[Uniform LSI for the 2D Coulomb Gas at the Diffusive Temperature Scale]{Uniform Logarithmic Sobolev Inequalities for the 2D Coulomb Gas at the Diffusive Temperature Scale}
\author{Matthew Rosenzweig}
\address{Matthew Rosenzweig, Department of Mathematical Sciences, Carnegie Mellon University, 7127 Wean Hall, 5000 Forbes Avenue, Pittsburgh, PA 15213, USA}
\email{mrosenz2@andrew.cmu.edu}
\urladdr{https://matthewrosenzweigwork-max.github.io/}
\thanks{M.R. was supported by NSF grants DMS-2441170, DMS-2345533, and DMS-2342349, and by OpenAI through their academic researcher initiative.}

\begin{document}
\raggedbottom

\begin{abstract}
For $N\ge2$ and $\beta>0$, consider the canonical 2D Coulomb gas ensemble
\begin{equation}
 \D\mathbb P_{N,\beta}(Z)
 =\mathsf{Z}_{N,\beta}^{-1}e^{-\beta|Z|^2/2}
   \prod_{i<j}|z_i-z_j|^{\beta/N}\,\D Z
 \qquad\text{on }\C^N.
\end{equation}
The factor $N^{-1}$ in the pair exponent places the ensemble at the diffusive, or high-temperature, scale.  We prove that, for every fixed inverse temperature $\beta>0$, this Gibbs measure satisfies a logarithmic Sobolev inequality (LSI) on the closed Dirichlet-form domain generated by collision-free compactly supported smooth functions, with a positive constant independent of $N$.  The LSI constant may be chosen uniformly when $\beta$ ranges over a compact subset of $(0,\infty)$.  A second main result establishes Poincar\'e and logarithmic Sobolev inequalities for planar Gaussian measures weighted by finite products of positive powers of distances to points, with constants depending only on the total variance of the underlying Gaussian measure and the total exponent, not on the number or locations of the points.  To prove the full-labeled Poincar\'e inequality, we decompose variance into permutation-invariant and label-dependent parts, controlled respectively by a weighted $\bar\partial$ estimate and by relative-coordinate conditioning with random transpositions; separately, one-site logarithmic Sobolev inequalities, a conditional entropy inequality, and a repulsive partition estimate give a defective logarithmic Sobolev inequality.  Rothaus tightening combines the resulting full-labeled Poincar\'e inequality with the defective logarithmic Sobolev inequality.  The two main theorems, for the weighted Gaussian measures and the labeled Coulomb gas on their stated closed form domains, have been formalized in Lean~4, together with all the mathematical results invoked by their formal proofs, including those taken from the literature.
\end{abstract}

\maketitle
\enlargethispage{7pt}

\section{Introduction}\label{sec:introduction}

\subsection{The model and main theorem}\label{sec:intro-model}

Identify $\R^2$ with $\C$ and write $\g(z)=-\log|z|$ for the two-dimensional logarithmic/Coulomb potential.  For $Z=(z_1,\ldots,z_N)\in\C^N$, denote the $N$-particle Hamiltonian/energy by
\begin{equation}\label{eq:hamiltonian}
 \mathcal H_N(Z)
 :=\frac12\sum_{i=1}^N|z_i|^2
   +\frac1N\sum_{1\le i<j\le N}\g(z_i-z_j).
\end{equation}
The canonical ensemble for the 2D Coulomb gas considered in this paper is
\begin{equation}\label{eq:law}
 \D\mathbb P_{N,\beta}(Z)
 =\mathsf{Z}_{N,\beta}^{-1}e^{-\beta\mathcal H_N(Z)}\,\D Z
 =\mathsf{Z}_{N,\beta}^{-1}e^{-\beta|Z|^2/2}
   \prod_{i<j}|z_i-z_j|^{\beta/N}\,\D Z,
\end{equation}
where the normalizing constant $\mathsf{Z}_{N,\beta}$ is the partition function.
Here and below, {canonical} means that the particle number $N$ is fixed.  The interaction is repulsive, and the density therefore vanishes at collisions.  The factor $N^{-1}$ in \eqref{eq:hamiltonian} places the ensemble at the diffusive temperature scale: the noise amplitude in the reversible overdamped dynamics is independent of $N$.  In the usual log-gas convention, \eqref{eq:law} corresponds to the $N$-dependent inverse temperature $\bar\beta=\beta/N$.  From the random-matrix perspective, the diffusive temperature scale considered here is therefore a high-temperature regime.  By contrast, the classical random-matrix regime keeps $\bar\beta$ of order one: the pair exponent is $\bar\beta$ and the confinement coefficient is $N\bar\beta$.  In the resulting random-matrix ensemble, the particles occupy a region of macroscopic scale one and have bulk spacing $N^{-1/2}$, with order-one local repulsion after microscopic rescaling.  For these conventions and scales, see \cite[Chap.~1, \S1.1, \S1.3 and Chap.~3, \S3.2]{Serfaty2024Lectures}.

At first sight, the pair exponent $\beta/N$ in \eqref{eq:law} may suggest that the interaction is a small perturbation of the noninteracting Gaussian product measure, which has logarithmic Sobolev constant $\beta$ in the convention below, independently of $N$, by the Gaussian logarithmic Sobolev inequality \cite[(1.2)]{Gross1975}.  This is misleading for an $N$-uniform functional inequality.  The interaction is summed over $N(N-1)/2$ pairs, while its density factor relative to the Gaussian product measure vanishes at collisions and is unbounded above.  Thus, the bounded-perturbation principle \cite{HolleyStroock1987} is unavailable.  In addition, the Hessian of $\g=-\log|\cdot|$ is unbounded near collisions and indefinite away from them, so a direct uniform-convexity argument in the spirit of Bakry--\'Emery \cite[Prop.~3 and Cor.~2]{BakryEmery1985} does not apply.  The desired inequalities must moreover hold for arbitrary observables on the full-labeled space, not only for functions of the empirical measure.

Let $\Delta_N$ denote the Vandermonde polynomial and let $a_N$ denote the pair exponent:
\begin{equation}\label{eq:collision-notation}
 \Delta_N(Z):=\prod_{1\le i<j\le N}(z_i-z_j),
 \qquad a_N:=\frac\beta N.
\end{equation}
Define the collision set and the collision-free configuration space by
\begin{equation}\label{eq:collision-set}
 \mathcal C_N:=\{Z\in\C^N:z_i=z_j\text{ for some }i\ne j\},
 \qquad \Omega_N:=\C^N\setminus\mathcal C_N=\{\Delta_N\ne0\}.
\end{equation}
For a probability measure $\mathsf P$ on a Euclidean space $\mathsf X$ and a nonnegative $g\in L^1(\mathsf P)$, define the entropy of $g$ with respect to $\mathsf P$ by
\begin{equation}
 \Ent_{\mathsf P}(g)
 :=\int_{\mathsf X} g\log g\,\D\mathsf P
 -\left(\int_{\mathsf X} g\,\D\mathsf P\right)
  \log\left(\int_{\mathsf X} g\,\D\mathsf P\right),
\end{equation}
with the convention $0\log0=0$.  For general background on entropy and logarithmic Sobolev inequalities, see \cite{aneInegalitesSobolevLogarithmiques2000,GuionnetZegarlinski2003}.  We write $\rho(\mathsf P)$ for the optimal logarithmic Sobolev constant in the convention
\begin{equation}\label{eq:lsi-convention}
 \Ent_{\mathsf P}(f^2)
 \le \frac2{\rho(\mathsf P)}\int_{\mathsf X}|\nabla f|^2\,\D\mathsf P.
\end{equation}
Using real Euclidean gradients on $\C^N\simeq\R^{2N}$, define the Dirichlet energy on $C_c^\infty(\Omega_N)$ by
\begin{equation}\label{eq:intro-form}
 \cE_{N,\beta}(f)
 :=\int_{\C^N}|\nabla f|^2\,\D\mathbb P_{N,\beta}.
\end{equation}
Define the associated form norm and closed Dirichlet-form domain by
\begin{equation}\label{eq:form-domain}
 \|f\|_{N,\beta}^2
 :=\int_{\C^N}|f|^2\,\D\mathbb P_{N,\beta}+\cE_{N,\beta}(f),
 \qquad
 \mathsf D_{N,\beta}
 :=\overline{C_c^\infty(\Omega_N)}^{\,\|\cdot\|_{N,\beta}}.
\end{equation}
We also write $\rho_{N,\beta}:=\rho(\mathbb P_{N,\beta})$ for the optimal constant in \eqref{eq:lsi-convention} over real-valued $f\in\mathsf D_{N,\beta}$.

\begin{theorem}[Uniform logarithmic Sobolev inequality]\label{thm:main}
For every fixed $\beta>0$, there exists $\rho_*(\beta)>0$ such that, for all $N\ge2$ and every real-valued $f\in\mathsf D_{N,\beta}$,
\begin{equation}\label{eq:main-lsi}
 \Ent_{\mathbb P_{N,\beta}}(f^2)
 \le \frac2{\rho_*(\beta)}\cE_{N,\beta}(f).
\end{equation}
Moreover, for every compact interval $I\Subset(0,\infty)$,
\begin{equation}\label{eq:compact-beta-rate}
 \inf_{\beta\in I}\inf_{N\ge2}\rho_{N,\beta}>0.
\end{equation}
\end{theorem}

The proof of \cref{thm:main} yields an explicit positive lower bound in terms of the weighted one-particle, permutation-invariant, relative-coordinate, and partition constants introduced below; see \Cref{sec:tightening}.  The estimates are first proved on the collision-free core for the unregularized Gibbs measure \eqref{eq:law}.  \Cref{sec:model-domain} proves that every collision stratum has zero weighted capacity and that \eqref{eq:intro-form} closes to the form used in \Cref{thm:main}.  No cutoff-removal theorem for the Gibbs Dirichlet form enters the proof.  \Cref{thm:main} applies to the full-labeled observable class.  The control of the LSI constant is locally uniform in $\beta$, but the theorem does not provide uniform control on all of $(0,\infty)$ nor controlled asymptotics at either endpoint.

\subsection{The weighted one-particle theorem}\label{sec:intro-weighted}

The singular logarithmic interaction produces, after conditioning, planar Gaussian measures weighted by products of positive powers of distances to points.  Fix $\gamma>0$ and $Q<\infty$.  Let $y_1,\ldots,y_m\in\C$ and $a_1,\ldots,a_m\ge0$ satisfy the total exponent bound $\sum_ja_j\le Q$, and combine coincident points.  Denote by $\nu$ the weighted Gaussian probability measure, with normalizing constant $Z$, given by
\begin{equation}
 \D\nu(z)
 :=Z^{-1}e^{-\gamma|z|^2/2}
   \prod_{j=1}^m|z-y_j|^{a_j}\,\D z.
\end{equation}

The proof of \Cref{thm:main} requires estimates uniform over all point configurations arising from conditioning.  The following theorem provides exactly this: its constants depend only on the total variance $2\gamma^{-1}$ of the underlying Gaussian measure (that is, the trace of its covariance) and the total exponent bound, not on the number, locations, or clustering of the points.

\begin{theorem}[Planar weighted Gaussian inequalities]\label{thm:weighted}
There exist finite constants $C_{\rm wt}(\gamma,Q)$ and $\rho_{\rm wt}(\gamma,Q)>0$, depending only on $(\gamma,Q)$, such that every such measure $\nu$ satisfies
\begin{align}
 \Var_\nu(f)&\le C_{\rm wt}(\gamma,Q)\int_{\C}|\nabla f|^2\,\D\nu,\label{eq:weighted-pi}\\
 \Ent_\nu(f^2)&\le\frac2{\rho_{\rm wt}(\gamma,Q)}\int_{\C}|\nabla f|^2\,\D\nu.\label{eq:weighted-lsi}
\end{align}
The inequalities hold for every real-valued $f$ in the closed Dirichlet-form domain generated by $C_c^\infty(\C\setminus\{y_j:a_j>0\})$.  The constants are uniform in the number and locations of the points and in the individual exponents, subject only to this total-exponent bound.
\end{theorem}

The proof of \Cref{thm:weighted} separates local control near the point singularities from global control at infinity supplied by the Gaussian confinement.  The local argument draws on planar quasiconformal geometry and strong-$A_\infty$ weights \cite{BonkLang2003,DavidSemmes1990,Semmes1993,Bjorn2001}.  These tools control the behavior near the zeros of the weight, uniformly in the point configuration; see \Cref{lem:strong-Ainf,cor:local-gauss}.  Uniform normalization and Gaussian tails are established in \Cref{lem:norm-tail}.  The remaining global control comes from a quadratic-moment estimate: on an annulus of radius $R$, the Gaussian score has size $\gamma R$, while the logarithmic force field can cancel it only on a set controlled by a weak-$L^2$ bound.  The weak-field estimate and planar Ladyzhenskaya inequality prevent an $H^1$ function from concentrating enough mass on that exceptional set; see \Cref{lem:weak-field,lem:lady,prop:uncertainty}.  These ingredients yield the global Poincar\'e and Gaussian-rate super-Poincar\'e inequalities\footnote{Whereas the Poincar\'e inequality controls variance by a fixed multiple of the Dirichlet energy, a super-Poincar\'e inequality controls the full squared $L^2$ norm, for every scale $s>0$, by $s$ times the energy plus $\beta(s)$ times the squared $L^1$ norm.  The small-$s$ behavior of $\beta$ is called the rate; the Gaussian rate proved below yields a defective logarithmic Sobolev inequality.  See \cite[\S3.3]{Wang2005} and \cite[(1.3)]{CGWW2009}.} in \Cref{prop:weighted-pi,prop:super-pi}; the latter gives a defective logarithmic Sobolev inequality, which is tightened using the former and the Rothaus inequality \eqref{eq:rothaus-centering} \cite{Rothaus1985} to obtain \eqref{eq:weighted-lsi}.

\subsection{Proof strategy for the main theorem}\label{sec:intro-method}

The proof of \Cref{thm:main} has two independent branches: a full-labeled Poincar\'e inequality and a defective logarithmic Sobolev inequality.  \Cref{thm:weighted} is used twice in the proof: for the one-site conditional logarithmic Sobolev inequality in the entropy branch and for the relative-coordinate conditional Poincar\'e inequality in the Poincar\'e branch.  Rothaus tightening combines the two branches at the end in \Cref{sec:tightening}.

The Poincar\'e branch begins with averaging over particle labels.  Because $\mathbb P_{N,\beta}$ is exchangeable, permutation symmetrization of an observable agrees with conditional expectation given the unordered configuration, and variance splits orthogonally into a permutation-invariant part and a part containing label information.  A weighted $\bar\partial$ estimate in the framework of H\"ormander and Demailly \cite{Hormander1965,DemaillyCADG}, together with permutation symmetry, controls the first term.  For the second, fix $i<j$, write $c=(z_i+z_j)/\sqrt2$ and $u=(z_i-z_j)/\sqrt2$ for the orthonormal center and relative coordinates, respectively, and set $A_k:=c/\sqrt2-z_k$ for $k\notin\{i,j\}$.  Freezing $c$ and all remaining particles gives a relative-coordinate density proportional to
\begin{equation}
 e^{-\beta|u|^2/2}
 \left|u\prod_{k\ne i,j}\left(A_k^2-\frac{u^2}{2}\right)\right|^{\beta/N}\,\D u.
\end{equation}
The polynomial has degree $2N-3$, and hence total exponent $\beta(2N-3)/N<2\beta$.  Thus, the Poincar\'e inequality in \Cref{thm:weighted} applies uniformly in the frozen configuration; the factorization and resulting estimate are proved in \Cref{lem:relative-poly,prop:swap}.  The spectral gap of the random-transposition chain \cite[(1.1), Cor.~4, and Rem.~1]{DiaconisShahshahani1981} then reduces the label-dependent variance to a sum over particle pairs.  Summing the corresponding relative-coordinate gradients contributes a factor $N$, which cancels the factor $N^{-1}$ from the transposition inequality and yields a bound independent of $N$.  This branch is carried out in \Cref{sec:full-pi}.

The orthogonal projection identity for variance does not itself furnish a corresponding decomposition of entropy.  The defective logarithmic Sobolev inequality is therefore proved without symmetrizing the density.  Freezing all variables except $z_i$ gives a one-site conditional density proportional to
\begin{equation}
 e^{-\beta|z|^2/2}\prod_{j\ne i}|z-z_j|^{\beta/N},
\end{equation}
whose total exponent is $\beta(N-1)/N<\beta$.  The logarithmic Sobolev inequality in \Cref{thm:weighted} therefore controls the sum of one-site conditional entropies uniformly in the conditioned configuration.  A relative-entropy form of Han--Shearer (see Han \cite{Han1978} and, for the Polish product-space formulation used here, Madiman and Tetali \cite[Thm.~V, Cor.~VII, and (21)--(22)]{MadimanTetali2010}) compares that sum with the product entropy.  The exact representation \eqref{eq:modulated} of the canonical Gibbs measure relative to the product thermal-equilibrium measure then isolates an interaction term controlled by the uniform repulsive partition estimate \Cref{thm:partition}, which follows from the correlation-moment and exponential-tail argument of Delgadino and Gvalani \cite[Theorem~2.5 and Appendix~A]{DelgadinoGvalani2025}, with the normalization specified in \Cref{sec:partition} and the accompanying supplement.  This produces the defective inequality \Cref{cor:defective} for arbitrary labeled densities; see \Cref{sec:defective}.  The full-labeled Poincar\'e inequality is not used in that argument and enters only in the final Rothaus centering step of \Cref{sec:tightening}.

\subsection{Relation to previous work and scope}\label{sec:intro-related}

Uniform logarithmic Sobolev inequalities for regular mean-field particle systems have been obtained through conditional criteria, flat convexity, transport, and free-energy methods; see, among others, \cite{OttoReznikoff2007,GuillinLiuWuZhang2022,DelgadinoGvalaniPavliotisSmith2023,KookZhangChewiErdogduLi2024,Wang2024,ChewiNitandaZhang2024,Monmarche2024,BauerschmidtBodineauDagallier2025}.  These results motivate the outer defective-LSI-plus-Poincar\'e structure.  Their regularity assumptions or worst-case conditional hypotheses, however, do not directly include the unregularized logarithmic interaction, whose gradient and Hessian diverge at collisions and whose Hessian is indefinite away from them.

Several works studied planar log gases at the same high-temperature or diffusive scale \cite{GarciaZelada2019,Lambert2021,AkemannByun2019}; their results concerned equilibrium asymptotics rather than Poincar\'e or logarithmic Sobolev inequalities.

Poincar\'e and logarithmic Sobolev inequalities on ordered chambers for a class of one-dimensional singular Gibbs measures were established by Chafa\"i and Lehec \cite{ChafaiLehec2020}. Related work of Guillin, Le Bris, and Monmarch\'e studies one-dimensional log and Riesz particle systems through ordered-particle couplings, with long-time and propagation-of-chaos estimates \cite{GuillinLeBrisMonmarche2023}.  The author and Serfaty also proved an $N$-uniform logarithmic Sobolev inequality for the one-dimensional Riesz gas under uniformly convex confinement, together with a modulated version for prescribed-reference Gibbs measures whenever the associated effective confinement remains uniformly convex \cite[Props.~3.3--3.4]{RosenzweigSerfaty2025}.  The modulated inequality then enters their criterion for generation of chaos \cite[Thm.~1.3 and \S3.2]{RosenzweigSerfaty2025}.  These one-dimensional arguments rely on ordering.  In the plane, particles can pass around one another without crossing the collision set, so there is no analogous global chamber decomposition.

For the quadratically confined planar Coulomb Gibbs measures defined in \eqref{eq:law} (equivalently, the invariant measures in \cite[Eq.~(1.11)]{BCF2018} with their inverse-temperature parameter $\beta_N$ set to $\beta N/2$), Bolley, Chafa\"i, and Fontbona proved a fixed-particle-number Poincar\'e inequality by an exponential-energy Lyapunov argument and proposed a Lyapunov/local-super-Poincar\'e route to a fixed-particle-number logarithmic Sobolev inequality, while noting the collision-domain difficulty \cite[Thm.~1.3, Lem.~5.2, and \S1.4.5]{BCF2018}.  The abstract implication was supplied by Cattiaux and Guillin \cite[Prop.~3.5]{CattiauxGuillin2017}.  The author's companion exposition \cite[Prop.~2.1]{RosenzweigBlogPartI} sketched how zero weighted capacity of the collision strata addresses this domain obstruction and how the resulting defective inequality is tightened using the fixed-particle-number Poincar\'e inequality; this argument gives no control uniform in the particle number.  Lu and Mattingly developed a collision-sensitive Lyapunov construction for kinetic Coulomb dynamics, explicitly motivated by the overdamped limit, where they observed that the potential energy is a natural Lyapunov function \cite[Eqs.~(1.5)--(1.6), \S3.1, Thm.~2.5, and Prop.~2.7]{LuMattingly2020}.  They proved kinetic weighted-total-variation geometric ergodicity, but the same collision-coercivity mechanism informs the fixed-particle-number overdamped LSI route.

In independent work appearing shortly before the first version of this paper, Becker and Menegaki proved uniform spectral-gap estimates for mean-field Coulomb gases at small effective coupling \cite{BeckerMenegaki2026}.\footnote{In forthcoming joint work, Matias Delgadino, Rishabh Gvalani, and the author remove the small-coupling restriction for the whole-space Coulomb Poincar\'e inequality in dimensions $\d\ge3$; see \Cref{sec:hilbert-schmidt-program}.}  Their Appendix~A treats the planar logarithmic gas: setting their parameters $\chi=2$ and $\beta_N=\beta N/2$ gives exactly \eqref{eq:law}, and their effective charge is $\beta(N-1)/N$.  Their Theorem~A.1 therefore gives a full-labeled $N$-uniform Poincar\'e inequality for sufficiently small fixed $\beta>0$.  Their argument combines a one-site ground-state estimate with a moving-pole bound and a Dobrushin--Wu criterion.  The present theorem gives a logarithmic Sobolev inequality for every fixed $\beta>0$; the weighted point-configuration theorem and the separation of the labeled Poincar\'e and entropy arguments allow us to dispense with a small-coupling restriction.

Recent independent preprints of Suzuki and Chafa\"i determined the sharp finite-particle Poincar\'e inequality for the complex Ginibre ensemble on the unlabelled, equivalently permutation-invariant, sector; Suzuki also determined the exact infinite-particle unlabelled spectral gap \cite{Suzuki2026Ginibre,Chafai2026Ginibre}.  After the standard spatial rescaling to unit macroscopic support, the Ginibre measure corresponds to the low-temperature choice $\beta_N=2N$ in the normalization \eqref{eq:law}, rather than to fixed diffusive inverse temperature.  Their arguments used a Vandermonde ground-state transform, holomorphic projection, and a sharp Gaussian $\bar\partial$ estimate, closely related to the complex-analytic ingredient used below.  The regimes and conclusions are nevertheless different: these works proved sharp Poincar\'e inequalities for symmetric observables at Ginibre scaling, whereas \Cref{thm:main} is a logarithmic Sobolev inequality for every fixed $\beta>0$ and every labeled observable.\footnote{Coincidentally, related conclusions were obtained by the author with AI assistance in unpublished notes \cite{RosenzweigFiniteGinibreNotes2026,RosenzweigInfiniteGinibreNotes2026} shortly before the announcement of these papers.}

In an expository predecessor to this paper \cite{RosenzweigBlogPartI}, we described a perturbative proof of the same full-labeled uniform-in-$N$ logarithmic Sobolev inequality at small inverse temperature, together with a refinement reaching $0<\beta<1$.  The finite range there came from reciprocal-weight restrictions in the one-site and relative-coordinate conditional inequalities.  \Cref{thm:weighted} removes precisely those two restrictions by allowing arbitrary finite total exponent.  The weighted $\bar\partial$ estimate, random-transposition argument, conditional entropy inequality, partition bound, and Rothaus tightening play the same structural roles.  The present paper is logically independent of that exposition and does not use it as a mathematical input.

The principal new analytic input in the proof of \Cref{thm:main} is \Cref{thm:weighted}, which is uniform over arbitrary finite point configurations at every finite total exponent.  The many-particle contribution is to use this theorem in the relative-coordinate Poincar\'e and one-site logarithmic Sobolev steps and to combine the resulting permutation-invariant and label-dependent Poincar\'e estimates with a direct conditional-entropy argument for the unregularized Gibbs measure on the full-labeled space.  The other ingredients are established tools whose exact roles are recorded at their points of use.

Our main theorem concerns the unregularized whole-space 2D Coulomb gas with logarithmic interaction and isotropic quadratic confinement, with diffusive pair exponent $\beta/N$, for every fixed $\beta>0$ and all labeled observables.  The proof relies essentially on planar strong-$A_\infty$ geometry, the weak-$L^2$/Ladyzhenskaya pairing, scalar complex factorization of the relative-coordinate conditionals, and pluriharmonicity of the logarithmic Vandermonde away from collisions.  It gives no automatic extension to general confinement, prescribed-reference ensembles, higher-dimensional Coulomb interactions, Riesz kernels, periodic models, optimal constants, a single logarithmic Sobolev constant uniform on the whole positive temperature half-line, or controlled endpoint asymptotics.  The present theorem is one nonperturbative component of a broader program on uniform logarithmic Sobolev inequalities and related functional inequalities for logarithmic/Riesz gases in the Hilbert--Schmidt regime.  The remaining parts of that program require different arguments; we discuss them in more detail in \Cref{sec:hilbert-schmidt-program}.

\subsection{Lean formalization and changes from version 1}\label{sec:intro-formalization}

The two main results, \Cref{thm:weighted,thm:main}, have been formalized in Lean~4, together with all the mathematical results invoked by their formal proofs, whether proved here or taken from the literature.  External results are formalized in the forms used in the proofs of the two main theorems.  The source is available in \cite{RosenzweigPlanarCoulombLean2026}.  The verification provenance records the transitive proof dependencies and acceptance by Lean's default kernel and the strict Nanoda checker, using only Lean's standard axioms of propositional extensionality, classical choice, and quotient soundness.  These checks apply to the formal statements and proofs; the accompanying supplement describes their relation to the manuscript.  The formalization does not include the external applications or the forthcoming results discussed in \Cref{sec:discussion}.

The autoformalization was carried out using Anthropic's Claude and OpenAI's ChatGPT/Codex under the author's close supervision.  The AI tools produced the formal code; the author directed the work, reviewed the mathematical content, and takes responsibility for the statements and arguments presented here.  Where autoformalization led to a different argument, the written proof has been revised to follow that argument.

Both main theorem statements and their parameter ranges are unchanged from version~1.  The revisions make explicit the domains, approximation arguments, and integrability conditions needed in their proofs.  Among these are the passage from smooth functions to the closed form domains, the compact-temperature argument, removal of weighted holomorphic poles, and the normalization and signed-energy input in the partition estimate.  The permutation-invariant Poincar\'e proposition now explicitly states its closed-domain hypothesis.  Section~1 of the accompanying supplement gives a fuller account of these changes and the scope of the formalization; Section~2 records the partition normalization and moment argument.  The literature comparison and the outlook in \Cref{sec:hilbert-schmidt-program} have also been updated, with the periodic and whole-space conclusions distinguished.

\subsection{Organization of the paper}\label{sec:intro-organization}

\Cref{sec:model-domain} establishes the closed Dirichlet form and center-of-mass factorization.  \Cref{sec:weighted} proves \Cref{thm:weighted}.  \Cref{sec:full-pi} proves the Poincar\'e inequality on the full-labeled space.  \Cref{sec:defective} proves the defective logarithmic Sobolev inequality, and \Cref{sec:tightening} completes the proof of \Cref{thm:main}.  \Cref{sec:discussion} places the theorem in the broader Coulomb/Riesz functional-inequality program.

\subsection{Statement on AI use}\label{sec:ai-use}

The author used generative AI tools (Anthropic's Claude and OpenAI's ChatGPT/Codex) during the development and preparation of this paper to identify potentially relevant literature, explore and test mathematical arguments, improve the exposition, check internal consistency and cross-references, and carry out the autoformalization described in \Cref{sec:intro-formalization}.  The author treated all AI-generated output as provisional material requiring independent verification and did not rely on it as an authority, checking literature suggestions against the relevant sources and independently working through mathematical suggestions before deciding whether to incorporate them into the paper.  The author made all final decisions concerning the manuscript and takes full responsibility for its contents.

\subsection*{Acknowledgements}

The author thanks Matias Delgadino and Rishabh Gvalani for engaging discussion related to this project. The author also thanks Simon Becker for sharing his work \cite{BeckerMenegaki2026}.

\section{Model and singular Dirichlet form}\label{sec:model-domain}\label{sec:domain}

This section establishes the analytic domain for the Gibbs measure associated with the singular logarithmic interaction.  We first prove finiteness of the partition function, closability on the collision-free core, and zero weighted capacity of the collision set, thereby identifying the closed Dirichlet-form domain.  We then separate the Gaussian center-of-mass coordinate from the relative coordinates.

\subsection{The closed Dirichlet form}

\begin{proposition}[Partition function and local closability]\label{prop:wellposed}
For every fixed $N\ge2$ and $\beta>0$, $0<\mathsf{Z}_{N,\beta}<\infty$.  The form $\cE_{N,\beta}$ defined in \eqref{eq:intro-form} on $C_c^\infty(\Omega_N)$ is closable in $L^2(\mathbb P_{N,\beta})$.
\end{proposition}

\begin{proof}
Since
\begin{equation}
 |z_i-z_j|\le(1+|z_i|)(1+|z_j|),
\end{equation}
we have
\begin{equation}\label{eq:gibbs-polynomial-bound}
 \prod_{i<j}|z_i-z_j|^{\beta/N}
 \le\prod_{i=1}^N(1+|z_i|)^{\beta(N-1)/N},
\end{equation}
and Gaussian integrability gives finiteness.  Positivity follows by integrating over a product of disjoint balls.

On every compact subset of $\Omega_N$, the density is smooth, positive, and bounded above and below.  If $f_k\to0$ in $L^2(\mathbb P_{N,\beta})$ and $\nabla f_k$ is Cauchy, testing the limiting vector field against compactly supported smooth vector fields on $\Omega_N$ shows that it vanishes.  Exhaustion of $\Omega_N$ gives closability.
\end{proof}

\leavevmode Local closability on the collision complement does not yet show that removing the collision set leaves the closed form unchanged.  The next proposition supplies this identification by proving zero weighted capacity.

\leavevmode Here and below, $\operatorname{Cap}_{\cE_{N,\beta},1}$ denotes the variational $1$-capacity associated with the closed form, using the form norm in \eqref{eq:form-domain}; see \cite[(2.1.1)--(2.1.3) and Lem.~2.1.1]{FukushimaOshimaTakeda2011}.

\begin{proposition}[Localized collision capacity and common core]\label{prop:capacity}
For fixed $N$ and $\beta$, every collision stratum, and hence the full collision set $\mathcal C_N$, has zero $(\cE_{N,\beta},1)$-capacity.  The closure of $C_c^\infty(\Omega_N)$ agrees with the closure obtained from compactly supported smooth functions on $\C^N$.
\end{proposition}

\begin{proof}
Fix a pair $i<j$ and a compact set $K$, and let $\alpha=\beta/N$.  On a compact neighborhood $K'$ of $K$, all factors except $|z_i-z_j|^\alpha$ are bounded above, hence
\begin{equation}
 \D \mathbb P_{N,\beta}\le C_{K'}|z_i-z_j|^\alpha\,\D Z.
\end{equation}
Choose a smooth cutoff $\phi_\varepsilon$ with $\phi_\varepsilon(r)=0$ for $r\le\varepsilon$, $\phi_\varepsilon(r)=1$ for $r\ge2\varepsilon$, and $|\phi_\varepsilon'|\le C/\varepsilon$.  In the normal complex coordinate $u=(z_i-z_j)/\sqrt2$,
\begin{equation}\label{eq:pair-cutoff-energy}
 \int_{K'}|\nabla\phi_\varepsilon(|z_i-z_j|)|^2\,\D \mathbb P_{N,\beta}
 \le C_{K',N,\beta}\varepsilon^\alpha.
\end{equation}
The squared $L^2(K',\mathbb P_{N,\beta})$ norm of $1-\phi_\varepsilon(|z_i-z_j|)$ is $O(\varepsilon^{\alpha+2})$.

Let $\Phi_\varepsilon$ be the product of the pair cutoffs over all unordered pairs.  Multiplying a compactly supported smooth function by $\Phi_\varepsilon$ removes all collision strata with vanishing form error.  Local mollification inside $\Omega_N$ yields the common-core assertion.

Choose smooth outer cutoffs $\chi_R$ equal to one on $B_R$, supported in $B_{2R}$, with $|\nabla\chi_R|\le C/R$.  Since $\mathbb P_{N,\beta}$ is a probability measure, $\chi_R\to1$ in the form norm.  By the common-core assertion and a diagonal choice, there are $v_n\in C_c^\infty(\Omega_N)$ converging to $1$ in that norm.  Then $u_n:=1-v_n$ equals one on the open complement of $\operatorname{supp}v_n$, which contains the entire collision set, and its form norm tends to zero.  These are admissible global capacity competitors, so $\operatorname{Cap}_{\cE_{N,\beta},1}(\mathcal C_N)=0$.  Every collision stratum has zero capacity by monotonicity.

\end{proof}

\leavevmode We next record the corresponding approximation of the constant function $1$ in the many-particle and weighted one-particle domains.  It permits subtraction of means within the closed domain and will be used in the two Rothaus centering steps.

\begin{lemma}[Constants belong to the closed domains]\label{lem:constants-domain}
The constant function $1$ belongs to $\mathsf D_{N,\beta}$.  More generally, for every weighted Gaussian measure of \Cref{thm:weighted}, the constant function belongs to its closed form domain.
\end{lemma}

\begin{proof}
In either case, let $\mathsf P$ denote the probability measure and let $\mathsf X$ denote its ambient space, $\C^N$ or $\C$, respectively.
Choose a radial cutoff $\chi_R$ equal to one on $B_R$, zero outside $B_{2R}$, with $|\nabla\chi_R|\le C/R$.  Gaussian-polynomial tails imply
\begin{equation}
 \|1-\chi_R\|_2^2\longrightarrow0,
 \qquad
 \int_{\mathsf X}|\nabla\chi_R|^2\,\D\mathsf P
 \le \frac{C}{R^2}\mathsf P(B_{2R}\setminus B_R)\longrightarrow0.
\end{equation}
For fixed $R$, multiply $\chi_R$ by the finite product of collision-avoidance cutoffs used in \Cref{prop:capacity}.  Their $L^2$ and energy errors vanish as the collision scale tends to zero.  A diagonal sequence gives collision-free compactly supported smooth approximants converging to $1$ in form norm.

For a measure $e^{-\gamma|z|^2/2}\prod_j|z-y_j|^{a_j}\D z$ of the class in \Cref{thm:weighted}, combine coincident points and discard zero exponents.  At a point with local exponent $a>0$, the energy of a radial avoidance cutoff is $O(\varepsilon^a)$.  The same radial-tail argument then proves $1$ belongs to the point-punctured closed domain.
\end{proof}

\subsection{Center-of-mass factorization}

\leavevmode Having identified the closed gradient form, we now separate the common translation mode from the relative coordinates.  Because the interaction depends only on particle differences, the center of mass is an exact planar Gaussian factor.

\begin{proposition}[Center-of-mass factorization]\label{prop:com}

Let
\begin{equation}
 c_N=N^{-1/2}\sum_{i=1}^Nz_i,
 \qquad
 r_k:=\frac{z_1+\cdots+z_k-kz_{k+1}}{\sqrt{k(k+1)}},
 \quad 1\le k\le N-1,
\end{equation}
and set $Z^{\rm rel}:=(r_1,\ldots,r_{N-1})$.  The map $Z\mapsto(c_N,Z^{\rm rel})$ is unitary and, in these coordinates,

\begin{equation}
 \mathbb P_{N,\beta}=\gamma_\beta^{(2)}\otimes \mathbb P_{N,\beta}^{\rm rel},
\end{equation}
where $\gamma_\beta^{(2)}(dc)=(\beta/2\pi)e^{-\beta|c|^2/2}\D c$ and $\mathbb P_{N,\beta}^{\rm rel}$ is the law of $Z^{\rm rel}$.  Let $\Omega_N^{\rm rel}\subset\C^{N-1}$ be the complement of the relative collision hyperplanes.  The relative form is the closure of
\begin{equation}
 \cE^{\rm rel}_{N,\beta}(v)=\int_{\C^{N-1}}|\nabla_rv|^2\,\D\mathbb P_{N,\beta}^{\rm rel},
 \qquad v\in C_c^\infty(\Omega_N^{\rm rel}),
\end{equation}
in the norm $(\|v\|_2^2+\cE^{\rm rel}_{N,\beta}(v))^{1/2}$; equivalently, one may close $C_c^\infty(\C^{N-1})$.  The notation $\rho(\mathbb P_{N,\beta}^{\rm rel})$ refers to the optimal rate on this full relative domain.  Consequently,
\begin{equation}
 \rho(\mathbb P_{N,\beta})=\min\{\beta,\rho(\mathbb P_{N,\beta}^{\rm rel})\}.
\end{equation}
\end{proposition}

\begin{proof}

The displayed coefficient vectors form an orthonormal basis of $\C^N$, while every difference functional $Z\mapsto z_i-z_j$ annihilates the diagonal direction.  Thus, the Gaussian density and Lebesgue measure factor, and the interaction depends only on $Z^{\rm rel}$.  This gives the product decomposition and the energy splitting $|\nabla_Zf|^2=|\nabla_cf|^2+|\nabla_rf|^2$ in these coordinates.

We first justify the relative common core.  For $v\in C_c^\infty(\C^{N-1})$, define its lift by $\widetilde v(c,r):=v(r)$.  This is a bounded smooth function with bounded gradient, although it need not have compact support.  With the outer cutoffs $\chi_R$ from \Cref{prop:capacity},
\begin{equation}\label{eq:relative-lift-cutoff}
 \|\chi_R\widetilde v-\widetilde v\|_{N,\beta}^2
 \le 2\int_{\{|Z|>R\}}\bigl(|\widetilde v|^2+|\nabla\widetilde v|^2\bigr)\,\D\mathbb P_{N,\beta}
 +\frac{C\|v\|_\infty^2}{R^2}
 \xrightarrow[R\to\infty]{}0.
\end{equation}
Thus \Cref{prop:capacity} gives $F_n\in C_c^\infty(\Omega_N)$ with $F_n\to\widetilde v$ in the full form norm.  Regard $F_n$ as a function of $(c,r)$ and write $\|\cdot\|_{\rm rel}$ for the relative form norm.  Fubini's theorem and $|\nabla_rF_n-\nabla_rv|\le|\nabla(F_n-\widetilde v)|$ give
\begin{equation}\label{eq:relative-slice-error}
 \int_{\C}\|F_n(c,\cdot)-v\|_{\rm rel}^2\,\D\gamma_\beta^{(2)}(c)
 \le\|F_n-\widetilde v\|_{N,\beta}^2\longrightarrow0.
\end{equation}
Since $\gamma_\beta^{(2)}$ is a probability measure, there is a center coordinate $c_n$ for which the integrand is no greater than its average.  The slice $v_n:=F_n(c_n,\cdot)$ belongs to $C_c^\infty(\Omega_N^{\rm rel})$ and converges to $v$ in the relative form norm.  Conversely, a relative core function extends by zero to a compactly supported smooth function on $\C^{N-1}$.  Taking closures proves the asserted equality of the two relative domains.

Tensorization \cite[Rem.~3.3]{Gross1975}, applied first to compact collision-free tests and the two factor inequalities, now gives the rate $\min\{\beta,\rho(\mathbb P_{N,\beta}^{\rm rel})\}$, with the usual passage through smaller admissible rates if needed.  The energy splitting has no loss of constants.  Core approximation and lower semicontinuity of entropy extend the inequality to $\mathsf D_{N,\beta}$.
Conversely, a compact smooth test of either factor lifts to a bounded smooth cylinder with bounded gradient.  The same outer-cutoff and common-core approximation puts this cylinder in $\mathsf D_{N,\beta}$.  Product integration preserves its entropy and gives exactly its factor energy.  Testing any admissible full-domain rate on these cylinders, and then closing the factor cores, bounds it by each factor's optimal rate.  The Gaussian optimum is $\beta$, proving the equality.

\end{proof}


\section{Weighted Gaussian inequalities with point singularities}\label{sec:weighted}\label{sec:weighted-points}

This section proves \Cref{thm:weighted} for the weighted planar Gaussian measures defined below.
Fix $\gamma>0$ and $Q<\infty$.  Let $y_1,\dots,y_m\in\C$ and $a_1,\dots,a_m\ge0$, and denote their total exponent by
\begin{equation}
 q:=\sum_{j=1}^m a_j\le Q.
\end{equation}
After combining coincident points, define the point-singularity weight $w$ and the associated weighted Gaussian probability measure $\nu$, with normalizing constant $Z$, by
\begin{equation}\label{eq:w-def}
 w(z):=\prod_{j=1}^m|z-y_j|^{a_j},
 \qquad
 \D\nu(z):=Z^{-1}e^{-\gamma|z|^2/2}w(z)\,\D z.
\end{equation}
At every point carrying positive total exponent, the weight $w$ extends continuously and vanishes; here, ``point singularity'' refers to the singularities of $\log w$ and its score, not to a blow-up of $w$ itself.

The dependence on the total variance of the underlying Gaussian measure is fixed by dilation.  Under $x=\sqrt\gamma z$, the pushforward of $\nu$ is a measure of the same class whose underlying Gaussian measure has total variance $2$, with points $\sqrt\gamma y_j$ and the same exponents; the configuration-independent power of $\gamma$ is absorbed by normalization.  Since the Dirichlet energy is multiplied by $\gamma$, we may and do choose the constants in \Cref{thm:weighted} so that
\begin{equation}\label{eq:weighted-scaling}
 C_{\rm wt}(\gamma,Q)=\gamma^{-1}C_{\rm wt}(1,Q),
 \qquad
 \rho_{\rm wt}(\gamma,Q)=\gamma\rho_{\rm wt}(1,Q).
\end{equation}

The proof combines local Sobolev control with estimates on the exterior region.  The strong-$A_\infty$ geometry established in \Cref{lem:strong-Ainf} yields the local Sobolev estimate \eqref{eq:local-sob-gauss} on bounded disks, uniformly in the point configuration.  Gaussian confinement is then used twice.  First, \Cref{lem:norm-tail} provides a disk, whose radius depends only on $\gamma$ and $Q$, containing at least half of the probability mass.  Second, the score estimate of \Cref{prop:uncertainty} yields the quadratic-moment form bound \Cref{cor:U-bound}, which controls the squared $L^2(\nu)$ mass of a test function outside large disks in terms of its Dirichlet energy and total squared $L^2(\nu)$ norm.  Together these estimates give the global Poincar\'e and Gaussian-rate super-Poincar\'e inequalities in \Cref{prop:weighted-pi,prop:super-pi}; the latter yields a defective logarithmic Sobolev inequality, which is tightened using \eqref{eq:rothaus-centering} and the Poincar\'e bound to obtain \eqref{eq:weighted-lsi}.

\subsection{Local quasiconformal and \texorpdfstring{strong-$A_\infty$}{strong-A-infinity} geometry}

For the continuous weight $w$ in \eqref{eq:w-def} and $x,y\in\R^2$, let $B_{xy}$ denote the Euclidean ball having the segment from $x$ to $y$ as a diameter, and define the weighted-volume gauge $\delta_w$ by
\begin{equation}
 B_{xy}:=B\!\left(\frac{x+y}{2},\frac{|x-y|}{2}\right),
 \qquad
 \delta_w(x,y):=\left(\int_{B_{xy}}w\right)^{1/2},
\end{equation}
and define the path metric
\begin{equation}
 d_w(x,y):=\inf_\gamma\int_\gamma w^{1/2}\,\D s,
\end{equation}
where the infimum is over rectifiable curves joining $x$ to $y$.  For this continuous weight, the strong-$A_\infty$ conclusion used below is doubling together with quantitative comparability of $d_w$ and $\delta_w$; see \cite{DavidSemmes1990,Semmes1993}.  This metric-geometric condition is stronger than, and is not synonymous with, membership in the ordinary Muckenhoupt class $A_\infty$.  For a refresher on the ordinary class, see \cite[\S7.3]{Grafakos2014}.

\leavevmode For a ball $B$, write $w(B):=\int_Bw\,\D z$, $f_{B,w}:=(\int_Bfw\,\D z)/w(B)$, and $f_{B,\nu}:=\nu(B)^{-1}\int_Bf\,\D\nu$ for the corresponding weighted volumes and averages.  Likewise, $\fint_Bh w$ and $\fint_Bh\,\D\nu$ denote integration against the respective normalized restrictions to $B$.

\begin{lemma}[Strong-$A_\infty$ local geometry]\label{lem:strong-Ainf}
The measure $w\D z$ is doubling, with doubling constant depending only on $Q$.  It supports a weak $(1,2)$ Poincar\'e inequality, with constants depending only on $Q$.\footnote{Here, ``weak'' means that the gradient term may be averaged over a fixed dilate $\lambda B$ of $B$.  This differs from the semigroup-theoretic weak Poincar\'e inequality involving a scale-dependent remainder term; see \cite[p.~176]{Bjorn2001} for the former; for the latter, see \cite[\S7.5]{BakryGentilLedoux2014} and \cite{RocknerWang2001}.}  There are $\kappa_Q>1$ and $C_Q<\infty$, and one may take $\Lambda_Q=11$, such that, for every real-valued locally Lipschitz function $f$ and every Euclidean ball $B=B(x,r)$ with $r>0$,
\begin{equation}\label{eq:local-sob-raw}
 \left(\fint_B|f-f_{B,w}|^{2\kappa_Q}w\right)^{1/(2\kappa_Q)}
 \le C_Qr\left(\fint_{\Lambda_QB}|\nabla f|^2w\right)^{1/2}.
\end{equation}
\end{lemma}

\begin{proof}
For $q=0$, this is Euclidean.  Assume $q>0$ and set $\rho=w^{1/2}$.  Distributionally, the conformal factor satisfies
\begin{equation}
 -\Delta\log\rho=-\pi\sum_ja_j\delta_{y_j}.
\end{equation}
After combining coincident sites and discarding zero exponents, the conformal factor near a site $y_j$ has the local representation
\begin{equation}\label{eq:local-cone-factor}
 \rho(z)=|z-y_j|^{a_j/2}e^{h(z)},\qquad
 h(z)=\frac12\sum_{k\ne j}a_k\log|z-y_k|,
\end{equation}
where $h$ is harmonic near $y_j$.  Thus this site has cone angle $2\pi(1+a_j/2)$.  We check that $(\C,d_w)$, obtained by adjoining the sites to the punctured conformal surface, is complete.  Set $R_0:=1+\sum_j|y_j|$, so that $\rho(z)\ge1$ when $|z|\ge R_0$.  Choose a smooth radial cutoff $\chi$ equal to one on $B_{R_0}$ and zero outside $B_{2R_0}$, with $|\nabla\chi|\le C/R_0$, and define
\begin{equation}\label{eq:metric-escape}
 \eta(z):=(1-\chi(z))|z|,
 \qquad |\nabla\eta(z)|\le C\rho(z).
\end{equation}
Indeed, $\eta$ is constant on $B_{R_0}$, its gradient is bounded on $B_{2R_0}\setminus B_{R_0}$, and $|\nabla\eta|=1$ outside $B_{2R_0}$.  Integrating \eqref{eq:metric-escape} along paths gives
\begin{equation}
 |\eta(z)-\eta(x)|\le Cd_w(x,z),
 \qquad |z|\le2R_0+|x|+Cd_w(x,z).
\end{equation}
Consequently, every metric ball is Euclidean bounded.  The local representation \eqref{eq:local-cone-factor}, together with positivity of $\rho$ away from the sites, shows that the completed metric induces the Euclidean topology.  Its closed balls are therefore compact, so the metric is complete.

The metric is flat away from the sites.  Its curvature measure $\omega$\footnote{The curvature measure is the signed measure that agrees with Gaussian curvature times the Riemannian area measure on the smooth part and assigns mass $2\pi-\alpha$ to a cone point of angle $\alpha$.  In a conformal coordinate with line element $\rho|\D z|$, it is $-\Delta\log\rho$ distributionally; see \cite[\S\S1--2]{BonkLang2003}.} consists of the cone deficits:
\begin{equation}\label{eq:cone-curvature-masses}
 \omega(\{y_j\})=2\pi-2\pi(1+a_j/2)=-\pi a_j,
 \qquad \omega^+(\C)=0,\qquad \omega^-(\C)=\pi q.
\end{equation}
Bonk and Lang's work \cite[Thm.~9.1]{BonkLang2003} gives a quasiconformal homeomorphism\footnote{Here, ``quasiconformal'' means bounded distortion: $F$ is an orientation-preserving $W_{\mathrm{loc}}^{1,2}$ homeomorphism satisfying $\lVert DF(z)\rVert_{\mathrm{op}}^2\le KJ_F(z)$ almost everywhere for some $K<\infty$.  Equivalently, its differential sends infinitesimal circles to ellipses of uniformly bounded eccentricity; see \cite[\S2.5]{AstalaIwaniecMartin2009}.} $F:\C\to\C$ and a constant $C=C(Q)$ such that the Jacobian determinant $J_F$ satisfies
\begin{equation}\label{eq:BL-J}
 C^{-1}J_F\le w\le CJ_F.
\end{equation}
Their geometric theorem \cite[Thm.~1.1]{BonkLang2003} gives an $L$-bi-Lipschitz equivalence of the completed conformal metric with the Euclidean plane, where $L=\sqrt{1+q/2}$.

Semmes's quantitative framework \cite[paragraph following Question~1.5 and definition following (1.8)]{Semmes1993} implies that a weight quantitatively comparable with the Jacobian of a planar quasiconformal homeomorphism is strong-$A_\infty$.  Pointwise comparability changes only the quantitative constants.  Hence, $w$ is strong-$A_\infty$ with data controlled by $Q$.
By Björn \cite[Corollary~8]{Bjorn2001}, the strong-$A_\infty$ property implies the asserted weak $(1,2)$ Poincar\'e inequality (together with doubling, this says that $w\D z$ is $2$-admissible in Björn's terminology).  Quantitative doubling gives constants $c_Q,s_Q>0$ such that, for nested balls $B_0\subset B$ with respective radii $r_0\le r$,
\begin{equation}
 \frac{w(B_0)}{w(B)}\ge c_Q\left(\frac{r_0}{r}\right)^{s_Q}.
\end{equation}
Choose $\kappa_Q>1$ sufficiently close to one that
\begin{equation}
 s_Q\left(\frac12-\frac1{2\kappa_Q}\right)\le1.
\end{equation}
Björn's result \cite[Thm.~7]{Bjorn2001}, with the same measure on both sides, gives the local Sobolev inequality.  Enlarging the gradient ball to $11B$ changes only $C_Q$ by doubling and gives \eqref{eq:local-sob-raw}.
\end{proof}

\begin{corollary}[Gaussian local Sobolev inequality]\label{cor:local-gauss}
There exists $C_{\gamma,Q}<\infty$, depending only on $(\gamma,Q)$, such that, for every real-valued locally Lipschitz function $f$ and every $R\ge1$,
\begin{equation}\label{eq:local-sob-gauss}
 \left(\fint_{B_R}|f-f_{B_R,\nu}|^{2\kappa_Q}\D\nu\right)^{1/(2\kappa_Q)}
 \le C_{\gamma,Q}e^{C_{\gamma,Q}R^2}R
 \left(\fint_{B_{\Lambda_QR}}|\nabla f|^2\D\nu\right)^{1/2}.
\end{equation}
\end{corollary}

\begin{proof}
On $B_{\Lambda_QR}$, the normalized restrictions of $w\D z$ and $\nu$ are mutually comparable, with density ratios bounded by $e^{C_{\gamma,Q}R^2}$; the global factor $Z^{-1}$ cancels upon normalization.  Consequently, with $p=2\kappa_Q$, the choice of average and H\"older's inequality give
\begin{equation}
 \left(\fint_{B_R}|f-f_{B_R,\nu}|^p\D\nu\right)^{1/p}
 \le 2\left(\fint_{B_R}|f-f_{B_R,w}|^p\D\nu\right)^{1/p}
 \le 2e^{C_{\gamma,Q}R^2}
 \left(\fint_{B_R}|f-f_{B_R,w}|^pw\right)^{1/p}.
\end{equation}
Applying \eqref{eq:local-sob-raw} and transferring the normalized gradient average on $B_{\Lambda_QR}$ from $w\D z$ to $\nu$ gives \eqref{eq:local-sob-gauss}, after enlarging $C_{\gamma,Q}$.
\end{proof}

\subsection{Normalization, tails, and score control}

The local weighted inequalities are insensitive to the locations of the points carrying positive exponent, but they do not prevent mass from escaping to infinity.  To obtain uniform global control, we first establish normalization and Gaussian tail bounds, and then control the logarithmic force field in the score to recover a quadratic-moment estimate.

\begin{lemma}[Uniform normalization and tails]\label{lem:norm-tail}
Set $\mathsf A=\prod_j(1+|y_j|)^{a_j}$.  There are $0<c_{\gamma,Q}\le C_{\gamma,Q}<\infty$ such that
\begin{equation}\label{eq:Z-comparison}
 c_{\gamma,Q}\mathsf A\le Z\le C_{\gamma,Q}\mathsf A.
\end{equation}
Moreover,
\begin{equation}\label{eq:gaussian-tail}
 \nu(|z|>R)
 \le C_{\gamma,Q}\int_{|z|>R}e^{-\gamma|z|^2/2}(1+|z|)^Q\,\D z.
\end{equation}
In particular, $\nu(B_{R_{\gamma,Q}})\ge1/2$ for a radius independent of the point configuration.
\end{lemma}

\begin{proof}
The upper bound follows from $|z-y|\le(1+|z|)(1+|y|)$.  For the lower bound, use
\begin{equation}
 \fint_{B_1}\log|z-y|\,\D z=
 \begin{cases}(|y|^2-1)/2,&|y|\le1,\\ \log|y|,&|y|\ge1
 \end{cases},
\end{equation}
which is bounded below by $\log(1+|y|)-(\log2+1/2)$.  Jensen's inequality on $B_1$ gives
\begin{equation}
 Z\ge\pi e^{-\gamma/2-(\log2+1/2)Q}\mathsf A.
\end{equation}
The tail assertion follows after division by this lower bound.
\end{proof}

Define the logarithmic force field $H$ and the negative total force (equivalently, negative score) $F$ by
\begin{equation}
 H(x):=\nabla\log w(x)=\sum_ja_j\frac{x-y_j}{|x-y_j|^2},
 \qquad F(x):=\gamma x-H(x).
\end{equation}

\begin{lemma}[Weak-$L^2$ logarithmic force estimate]\label{lem:weak-field}
There is a universal $C$ such that
\begin{equation}\label{eq:weak-field}
 |\{x:|H(x)|>t\}|\le Cq^2t^{-2},\qquad t>0.
\end{equation}
\end{lemma}

\begin{proof}
Let $\eta=\sum_ja_j\delta_{y_j}$.  Denote its centered Hardy--Littlewood maximal function (over balls) by
\begin{equation}
 M\eta(x):=\sup_{r>0}\frac{\eta(B(x,r))}{\pi r^2}.
\end{equation}
Splitting the Riesz potential $I_1\eta(x):=\int_{\C}|x-y|^{-1}\D\eta(y)$ at radius $R$ gives
\begin{equation}
 I_1\eta(x)\le CM\eta(x)R+q/R.
\end{equation}
Optimization yields $I_1\eta\le C\sqrt{qM\eta}$.  The same covering argument as in \cite[Lemma~2.1.5 and the proof of Thm.~2.1.6]{Grafakos2014} gives the weak $(1,1)$ maximal inequality for the finite measure $\eta$, namely $|\{M\eta>\lambda\}|\le Cq/\lambda$, and proves \eqref{eq:weak-field}.
\end{proof}

\leavevmode The weak-$L^2$ estimate controls the area of the large-field set.  To control the $L^2$ mass of a test function on that set, we use the following planar Ladyzhenskaya inequality \cite[Chap.~I, (4.8)]{FoiasEtAl2001}.

\begin{lemma}[Planar Ladyzhenskaya inequality]\label{lem:lady}
There is a universal constant $C<\infty$ such that, for every $v\in C_c^\infty(\R^2)$,
\begin{equation}
 \|v\|_{L^4}^2\le C\|v\|_{L^2}\|\nabla v\|_{L^2}.
\end{equation}
\end{lemma}

The next estimate is the global confinement step.  On an annulus $|x|\simeq R$, the Gaussian part of the score has size $\gamma R$.  Substantial cancellation can occur only where the logarithmic force field $H$ is comparably large; \Cref{lem:weak-field} makes that exceptional set of order $R^{-2}$.  On its complement, the score controls $R^2|\psi|^2$, while on the exceptional set \Cref{lem:lady} prevents an $H^1$ function from concentrating enough mass to destroy the estimate.

\begin{proposition}[Quadratic-moment estimate from score control]\label{prop:uncertainty}
There is $C_{\gamma,Q}<\infty$ such that every $\psi\in C_c^\infty(\C\setminus\{y_j\})$ satisfies
\begin{equation}\label{eq:uncertainty}
 \int_{\C}|x|^2|\psi|^2\,\D x
 \le C_{\gamma,Q}\left(
 \int_{\C}|\nabla\psi|^2+\int_{\C}|F|^2|\psi|^2+\int_{\C}|\psi|^2\right).
\end{equation}
\end{proposition}

\begin{proof}
Let $R_k:=2^kR_0$ be the dyadic radii, and let $\mathcal A_k:=\{R_k\le|x|<2R_k\}$ and $\mathcal A_k^*:=\{R_k/2<|x|<4R_k\}$ denote the associated annuli and enlarged annuli.  Define the small-score exceptional set $E_k\subset\mathcal A_k$ by
\begin{equation}
 E_k:=\{x\in\mathcal A_k:|F(x)|<\gamma R_k/2\}.
\end{equation}
On $E_k$, $|H|>\gamma R_k/2$, so \eqref{eq:weak-field} gives $|E_k|\le C Q^2/(\gamma^2R_k^2)$.  On $\mathcal A_k\setminus E_k$,
\begin{equation}
 R_k^2\int_{\mathcal A_k\setminus E_k}|\psi|^2
 \le\frac4{\gamma^2}\int_{\mathcal A_k}|F|^2|\psi|^2.
\end{equation}
Choose $\chi_k\in C_c^\infty(\mathcal A_k^*)$, $\chi_k=1$ on $\mathcal A_k$, $|\nabla\chi_k|\le C/R_k$.  Hölder and \Cref{lem:lady} imply
\begin{equation}
 R_k^2\int_{E_k}|\psi|^2
 \le C_{\gamma,Q}R_k\|\chi_k\psi\|_2\|\nabla(\chi_k\psi)\|_2.
\end{equation}
For every $\delta>0$, Young's inequality gives
\begin{equation}
 R_k^2\int_{E_k}|\psi|^2
 \le\delta R_k^2\|\psi\|_{L^2(\mathcal A_k^*)}^2
 +C_{\gamma,Q,\delta}\|\nabla\psi\|_{L^2(\mathcal A_k^*)}^2
 +C_{\gamma,Q}\|\psi\|_{L^2(\mathcal A_k^*)}^2.
\end{equation}
The expanded annuli have bounded overlap and $\sum_kR_k^2\1_{\mathcal A_k^*}\le C|x|^2$ off a fixed core.  Sum, take $R_0$ large, then choose $\delta$ small and absorb the quadratic-moment term.
\end{proof}

\leavevmode The preceding proposition controls the quadratic moment of $\psi$ with respect to Lebesgue measure by the Lebesgue integrals of $|\nabla\psi|^2$, $|F|^2|\psi|^2$, and $|\psi|^2$.  Under the ground-state substitution in the proof below, identity \eqref{eq:ground-state} converts this control into the weighted moment bound needed for the global Poincar\'e and super-Poincar\'e inequalities.

\begin{corollary}[Quadratic-moment form bound]\label{cor:U-bound}
There is $C_U(\gamma,Q)<\infty$ such that, on the closed form domain,
\begin{equation}\label{eq:U-bound}
 \int_{\C}|x|^2f^2\,\D\nu
 \le C_U(\gamma,Q)\left(\int_{\C}|\nabla f|^2\,\D\nu+\int_{\C} f^2\,\D\nu\right).
\end{equation}
\end{corollary}

\begin{proof}
For a smooth compactly supported function $f$ whose support avoids the points $y_j$ with $a_j>0$, let $U$ denote the negative logarithm of the unnormalized density and let $\psi$ denote the ground-state transform:
\begin{equation}
 U=\frac\gamma2|x|^2-\log w,
 \qquad \psi=Z^{-1/2}fe^{-U/2}.
\end{equation}
Since $\Delta U=2\gamma$ off the points, integration by parts gives
\begin{equation}\label{eq:ground-state}
 \int_{\C}|\nabla f|^2\,\D\nu
 =\int_{\C}\left(|\nabla\psi|^2+\frac14|F|^2|\psi|^2-\gamma|\psi|^2\right)\D x.
\end{equation}
Apply \Cref{prop:uncertainty}.  The point-avoidance and closure argument in \Cref{lem:constants-domain} extends the resulting inequality to the closed domain.  Distributionally the omitted point masses contribute nonnegatively; because the support of $f$ avoids the points, no point boundary is present.
\end{proof}

\subsection{Global Poincar\'e and logarithmic Sobolev inequalities}

\leavevmode We now combine the local Sobolev estimate with the quadratic-moment bound.  Together they first yield a global Poincar\'e inequality and then a Gaussian-rate super-Poincar\'e inequality; the latter gives a defective logarithmic Sobolev inequality, which Rothaus centering tightens using the Poincar\'e bound.

\begin{proposition}[Global Poincar\'e inequality]\label{prop:weighted-pi}
Inequality \eqref{eq:weighted-pi} holds.
\end{proposition}

\begin{proof}
Choose $R\ge R_{\gamma,Q}$ large enough that the coefficient below can be absorbed, and let $c=f_{B_R,\nu}$.  The local Poincar\'e consequence of \eqref{eq:local-sob-gauss} gives
\begin{equation}
 \int_{B_R}|f-c|^2\,\D\nu
 \le C_{\gamma,Q,R}\int_{\C}|\nabla f|^2\,\D\nu.
\end{equation}
Since $c$ is a permissible comparison constant,
\begin{equation}
 \Var_\nu(f)\le\int_{\C}|f-c|^2\,\D\nu=:M.
\end{equation}
By \Cref{lem:constants-domain}, $f-c$ belongs to the closed form domain, so \Cref{cor:U-bound} applies to it.  Hence,
\begin{equation}
 \int_{B_R^c}|f-c|^2\,\D\nu
 \le\frac{C_U}{R^2}\left(\int_{\C}|\nabla f|^2\,\D\nu+M\right).
\end{equation}
Choose $R^2\ge2C_U$.  The last $M/2$ term is absorbed, proving the Poincar\'e inequality first on the core and then by closure.
\end{proof}

\begin{proposition}[Gaussian-rate super-Poincar\'e inequality]\label{prop:super-pi}
There are $A_{\gamma,Q},B_{\gamma,Q}<\infty$ such that, for all $s>0$,
\begin{equation}\label{eq:super-pi}
 \int_{\C} f^2\,\D\nu
 \le s\int_{\C}|\nabla f|^2\,\D\nu
 +A_{\gamma,Q}e^{B_{\gamma,Q}/s}\left(\int_{\C}|f|\,\D\nu\right)^2.
\end{equation}
\end{proposition}

\begin{proof}
For $g$ supported in $B_{2R}$, apply the Gaussian local Sobolev inequality \eqref{eq:local-sob-gauss} on $B_{2R}$ and interpolate between $L^1(\nu)$ and $L^{2\kappa_Q}(\nu)$ with parameter
\begin{equation}
 \theta_Q=\frac{\kappa_Q}{2\kappa_Q-1},
 \qquad \frac12=(1-\theta_Q)+\frac{\theta_Q}{2\kappa_Q}.
\end{equation}
Young's inequality then gives, for $u\in(0,1]$,
\begin{equation}
 \int_{\C} g^2\,\D\nu
 \le u\int_{\C}|\nabla g|^2\,\D\nu
 +C_{\gamma,Q}e^{C_{\gamma,Q}R^2}(1+u^{-\eta_Q})
   \left(\int_{\C}|g|\,\D\nu\right)^2,
\end{equation}
where $\eta_Q=\theta_Q/(1-\theta_Q)=\kappa_Q/(\kappa_Q-1)$.  Let $\chi$ equal one on $B_R$, vanish outside $B_{2R}$, and satisfy $|\nabla\chi|\le2/R$.  Set $g=\chi f$ and $h=(1-\chi)f$.  Since $h=0$ on $B_R$, \eqref{eq:U-bound} gives, for large $R$,
\begin{equation}
 \int_{\C} h^2\,\D\nu\le\frac{2C_U}{R^2}\int_{\C}|\nabla h|^2\,\D\nu.\label{eq:spi-outer}
\end{equation}
Since $f=g+h$, we also have
\begin{equation}
 \int_{\C}f^2\,\D\nu
 \le2\int_{\C}g^2\,\D\nu+2\int_{\C}h^2\,\D\nu.
\end{equation}
The cutoff-gradient bounds give
\begin{equation}
 \int_{\C}|\nabla g|^2\,\D\nu
 +\int_{\C}|\nabla h|^2\,\D\nu
 \le 2\int_{\C}|\nabla f|^2\,\D\nu
 +\frac{C}{R^2}\int_{\C} f^2\,\D\nu.
\end{equation}
We absorb the cutoff-gradient contribution proportional to $M:=\int f^2\D\nu$.  Write $E:=\int|\nabla f|^2\D\nu$ and
\begin{equation}
 D_{R,u}:=C_{\gamma,Q}e^{C_{\gamma,Q}R^2}(1+u^{-\eta_Q}).
\end{equation}
The local estimate for $g$ and the outer estimate \eqref{eq:spi-outer}, followed by the cutoff-gradient bound, give the single chain
\begin{align}\label{eq:spi-cutoff-absorption}
 M&\le 2\int g^2\D\nu+2\int h^2\D\nu\notag\\
 &\le 2u\int|\nabla g|^2\D\nu+\frac{4C_U}{R^2}\int|\nabla h|^2\D\nu
       +2D_{R,u}\left(\int|f|\D\nu\right)^2\notag\\
 &\le\max\{2u,4C_U/R^2\}\left(2E+\frac C{R^2}M\right)
       +2D_{R,u}\left(\int|f|\D\nu\right)^2.
\end{align}
Choose $R^2=L/s$ and $u=\varepsilon s$.  The coefficient of $E$ is at most $2s\max\{2\varepsilon,4C_U/L\}$, and the coefficient of $M$ is at most
$Cs^2\max\{2\varepsilon,4C_U/L\}/L$.  First, choose $L$ large and $\varepsilon$ small so that the former is at most $s/2$.  For all sufficiently small $s>0$, the latter is at most $1/2$, $R$ is large enough for the outer estimate, and $u\le1$.  Moreover, $D_{R,u}\le Ae^{B/s}$ after increasing $A$ and $B$.  Moving the remaining $\int_{\C}f^2\D\nu$ term to the left shows that there is a sufficiently small $s_0>0$ such that \eqref{eq:super-pi} holds with $s$ replaced by $\min(s,s_0)$.  After increasing $A$ and $B$, the asserted bound follows.
\end{proof}

We shall use twice the following centering inequality of Rothaus: for every probability measure $\mathsf P$ on a measurable space $\mathsf X$, writing $\E_{\mathsf P}[f]:=\int_{\mathsf X} f\,\D\mathsf P$, and every real-valued $f$ for which the right-hand side is finite,
\begin{equation}\label{eq:rothaus-centering}
 \Ent_{\mathsf P}(f^2)\le \Ent_{\mathsf P}((f-\E_{\mathsf P}[f])^2)+2\Var_{\mathsf P}(f).
\end{equation}
We apply this inequality to bounded $f$ before passing the resulting tight inequalities to the form domains.  See Rothaus \cite{Rothaus1985}; for this probability-space formulation, see Barthe and Roberto \cite[Lemma~4]{BartheRoberto2003}.

\begin{proof}[Completion of \Cref{thm:weighted}]
Let $Y:=\{y_j:a_j>0\}$, let $\mathcal C_\nu:=C_c^\infty(\C\setminus Y;\R)$, and denote by $D(\cE_\nu)$ its closure in the norm $(\nu(f^2)+\cE_\nu(f))^{1/2}$, where $\cE_\nu(f):=\int|\nabla f|^2\D\nu$.  We first derive a defective logarithmic Sobolev inequality on $\mathcal C_\nu$ from \Cref{prop:super-pi}.  This is the Gaussian-rate implication of \cite[Corollary~3.3.4(1), $\delta=1$]{Wang2005}; see also \cite[discussion following (1.4)]{CGWW2009}.  To derive the defective logarithmic Sobolev inequality from \eqref{eq:super-pi}, we apply the super-Poincar\'e estimate to successive amplitude truncations and sum the resulting bounds.  We follow \cite[proof of Theorem~3.3.3]{Wang2005}, using smooth truncations that preserve $\mathcal C_\nu$.

Choose a fixed smooth transition $\vartheta$ equal to zero on $(-\infty,0]$ and to one on $[1,\infty)$, with $0\le\vartheta\le1$ and $|\vartheta'|\le K$.  For $b>0$, set $T_b(t):=b\vartheta((t-b)/b)$.  Then $T_b=0$ on $(-\infty,b]$, $T_b=b$ on $[2b,\infty)$, and $T_b\circ g\in\mathcal C_\nu$ for $g\in\mathcal C_\nu$.  The chain rule gives
\begin{equation}\label{eq:smooth-truncation-energy}
 \cE_\nu(T_b(g))\le K^2\int_{\{b<g<2b\}}|\nabla g|^2\D\nu.
\end{equation}
Write $A=A_{\gamma,Q}>0$ and $B=B_{\gamma,Q}>0$ in \eqref{eq:super-pi}, increasing them if necessary.  For $g\in\mathcal C_\nu$ with $\nu(g^2)=1$, set
\begin{equation}
 b_k:=2^{k+1},\qquad s_k:=\frac{B}{(k+1)\log2},\qquad k\ge0.
\end{equation}
The pointwise bound $0\le T_b(t)\le b^{-1}t^2$ gives $\nu(|T_b(\pm g)|)\le b^{-1}$.  Since $e^{B/s_k}=b_k$, applying \eqref{eq:super-pi} to $T_{b_k}(\pm g)$ yields
\begin{equation}\label{eq:dyadic-spi-layer}
 (k+1)\nu((T_{b_k}(\pm g))^2)
 \le\frac B{\log2}\cE_\nu(T_{b_k}(\pm g))
       +A\frac{k+1}{2^{k+1}}.
\end{equation}
For each sign the gradient bands in \eqref{eq:smooth-truncation-energy} are disjoint, and $\sum_{k\ge0}(k+1)2^{-k-1}=2$.  Hence, for every integer $J\ge0$,
\begin{equation}
 \sum_{k=0}^{J-1}(k+1)\nu((T_{b_k}(\pm g))^2)
 \le\frac{BK^2}{\log2}\cE_\nu(g)+2A.
\end{equation}
Choose $J$ so large that $|g|\le2^{J+2}$.  The elementary entropy-layer estimate is
\begin{equation}\label{eq:dyadic-entropy-layer}
 t^2\log t^2\le4\log2\,t^2
 +96\log2\sum_{k=0}^{J-1}(k+1)
       \bigl(T_{b_k}(t)^2+T_{b_k}(-t)^2\bigr),
 \qquad |t|\le2^{J+2}.
\end{equation}
For $|t|\le4$, the first term suffices.  Otherwise choose $k<J$ such that $2b_k<|t|\le4b_k$; one of $T_{b_k}(t)$ and $T_{b_k}(-t)$ equals $b_k$, while
$t^2\log t^2\le32(k+3)\log2\,b_k^2\le96(k+1)\log2\,b_k^2$.
Integrating \eqref{eq:dyadic-entropy-layer} and summing \eqref{eq:dyadic-spi-layer} for both signs gives
\begin{equation}
 \Ent_\nu(g^2)\le192BK^2\cE_\nu(g)+(4+384A)\log2.
\end{equation}
If $\nu(f^2)=0$, then $f=0$ almost everywhere and its entropy vanishes.  Otherwise set $g=f/\sqrt{\nu(f^2)}$ and use homogeneity of entropy and energy.  We obtain
\begin{equation}\label{eq:weighted-defective-lsi}
 \Ent_\nu(f^2)
 \le A_{\rm wt}(\gamma,Q)\int_{\C}|\nabla f|^2\,\D\nu
 +B_{\rm wt}(\gamma,Q)\int_{\C}f^2\,\D\nu,
\end{equation}
with, for example, $A_{\rm wt}=192BK^2$ and $B_{\rm wt}=(4+384A)\log2$.
For $f\in D(\cE_\nu)$, choose $f_n\in\mathcal C_\nu$ converging to $f$ in form norm.  Then $f_n^2\to f^2$ in $L^1(\nu)$ and
\begin{equation}\label{eq:weighted-defective-closure}
 \Ent_\nu(f^2)\le\liminf_{n\to\infty}\Ent_\nu(f_n^2)
 \le A_{\rm wt}\cE_\nu(f)+B_{\rm wt}\nu(f^2).
\end{equation}
By \Cref{lem:constants-domain}, $1\in D(\cE_\nu)$ with zero energy, so $f-\E_\nu[f]\in D(\cE_\nu)$ and $\cE_\nu(f-\E_\nu[f])=\cE_\nu(f)$.  For $f$ in the smooth core, apply \eqref{eq:rothaus-centering} with $\mathsf P=\nu$, \eqref{eq:weighted-defective-closure} to $f-\E_\nu[f]$, and \Cref{prop:weighted-pi} to obtain
\begin{equation}
 \Ent_\nu(f^2)
 \le\left[A_{\rm wt}+(B_{\rm wt}+2)C_{\rm wt}\right]
 \int_{\C}|\nabla f|^2\,\D\nu.
\end{equation}
Form approximation and lower semicontinuity of entropy extend this tight inequality to the full domain.  This proves \eqref{eq:weighted-lsi} with
\begin{equation}
 \rho_{\rm wt}(\gamma,Q)
 =\frac2{A_{\rm wt}(\gamma,Q)+(B_{\rm wt}(\gamma,Q)+2)C_{\rm wt}(\gamma,Q)}.
\end{equation}
\end{proof}


\section{Poincar\'e inequality on the full-labeled space}\label{sec:full-pi}

\begin{theorem}[Poincar\'e inequality for the full-labeled Gibbs measure]\label{thm:full-pi}
For every fixed $\beta>0$, every $N\ge2$, and every real-valued $f\in\mathsf D_{N,\beta}$,
\begin{equation}\label{eq:full-pi}
 \Var_{\mathbb P_{N,\beta}}(f)
 \le C_P^{\rm all}(\beta)\cE_{N,\beta}(f),
 \qquad C_P^{\rm all}(\beta)=C_{\rm sym}(\beta)+C_{\rm rel}(\beta)<\infty.
\end{equation}
The component constants $C_{\rm sym}(\beta)$ and $C_{\rm rel}(\beta)$ are defined below and depend only on $\beta$; in particular, $C_P^{\rm all}(\beta)$ is independent of $N$.
\end{theorem}

For $f\in L^2(\mathbb P_{N,\beta})$, define the permutation symmetrizer by
\begin{equation}\label{eq:symmetrizer}
 \sfS f(Z):=\frac1{N!}\sum_{\sigma\in S_N}
 f(z_{\sigma(1)},\ldots,z_{\sigma(N)}).
\end{equation}
Since $\mathbb P_{N,\beta}$ is exchangeable, $\sfS$ is the orthogonal projection onto the permutation-invariant subspace; equivalently, it is conditional expectation given the unordered configuration.  Thus,
\begin{equation}
 F(Z)=\sum_{i=1}^N\varphi(z_i)\quad\Longrightarrow\quad \sfS F=F,
\end{equation}
whereas
\begin{equation}
 f(Z)=\varphi(z_1)-\varphi(z_2)\quad\Longrightarrow\quad \sfS f=0.
\end{equation}
The corresponding variance decomposition is
\begin{equation}\label{eq:variance-permutation-decomposition}
 \Var_{\mathbb P_{N,\beta}}(f)
 =\Var_{\mathbb P_{N,\beta}}(\sfS f)+\|f-\sfS f\|_2^2.
\end{equation}
The two terms require different arguments.  Weighted $\bar\partial$ estimates and symmetry control the first.  A two-particle relative-coordinate conditional and random transpositions control the second.  The displayed orthogonal decomposition does not itself furnish a corresponding decomposition of entropy, so \Cref{sec:defective} uses a separate argument.

\subsection{Permutation-invariant functions via weighted \texorpdfstring{$\bar\partial$}{d-bar}}

\leavevmode We first control the permutation-invariant term in the variance decomposition.  A weighted $\bar\partial$ estimate bounds the component orthogonal to the holomorphic subspace; pole removal and common-phase invariance control the holomorphic component when $\beta/N\le2$.  For the finitely many smaller particle numbers for which $\beta/N>2$, a compact-resolvent argument supplies a fixed-$N$ Poincar\'e gap.

Define the effective potential $\Phi_{N,\beta}$ by
\begin{equation}
 \Phi_{N,\beta}(Z):=\frac\beta2|Z|^2-a_N\log|\Delta_N(Z)|.
\end{equation}
On the collision complement $\Omega_N$, $\mathbb P_{N,\beta}$ is proportional to $e^{-\Phi_{N,\beta}}\D Z$.

\leavevmode With $z_j=x_j+iy_j$, define the Wirtinger derivatives $\partial_j$ and $\bar\partial_j$ by
\begin{equation}
 \partial_j:=\frac12\left(\frac\partial{\partial x_j}-i\frac\partial{\partial y_j}\right),
 \qquad
 \bar\partial_j:=\frac12\left(\frac\partial{\partial x_j}+i\frac\partial{\partial y_j}\right).
\end{equation}
\leavevmode For a real-valued $u\in C^2(\Omega_N)$, the Levi form is represented by the Hermitian matrix $(\partial_j\bar\partial_k u)_{j,k}$.  The function $u$ is pluriharmonic if this matrix vanishes, plurisubharmonic if it is positive semidefinite, and strictly plurisubharmonic if it is positive definite.  An exhaustion has relatively compact strict sublevel sets; a complex manifold admitting a smooth plurisubharmonic exhaustion is weakly pseudoconvex \cite[Chap.~I, Defs.~6.12 and~6.13(a); Chap.~VIII, Def.~5.1]{DemaillyCADG}.

\begin{lemma}[Levi form and weak pseudoconvexity]\label{lem:levi}
On $\Omega_N$,
\begin{equation}
 (\partial_j\bar\partial_k\Phi_{N,\beta})_{j,k}=\frac\beta2I_N.
\end{equation}
Moreover, $|Z|^2+|\Delta_N(Z)|^{-2}$ is a smooth strictly plurisubharmonic exhaustion of $\Omega_N$.
In particular, $\Omega_N$ is weakly pseudoconvex.
\end{lemma}

\begin{proof}
Since $\Delta_N$ is nonvanishing and holomorphic on $\Omega_N$, $\log|\Delta_N|$ is locally the real part of a holomorphic logarithm, hence is pluriharmonic; this gives the Levi-form identity.  Moreover, $\Delta_N^{-1}$ is holomorphic on $\Omega_N$, so $|\Delta_N|^{-2}=|\Delta_N^{-1}|^2$ is plurisubharmonic.  Adding $|Z|^2$ makes this function strictly plurisubharmonic.  It diverges as $|Z|\to\infty$ or as $Z$ approaches $\mathcal C_N$, and hence is an exhaustion of $\Omega_N$.
\end{proof}

Define the weighted holomorphic space $\mathcal A^2_{N,\beta}$ by
\begin{equation}
 \mathcal A^2_{N,\beta}:=\{h\in L^2(\mathbb P_{N,\beta};\C):h\text{ is holomorphic on }\Omega_N\}.
\end{equation}
Let $Q_{N,\beta}$ denote the orthogonal projection onto this closed space.

\begin{proposition}[Distance to the weighted holomorphic space]\label{prop:dbar}
For $f\in C_c^\infty(\Omega_N;\C)$,
\begin{equation}\label{eq:dbar}
 \|f-Q_{N,\beta}f\|_2^2
 \le\frac2\beta\sum_{j=1}^N\|\bar\partial_jf\|_2^2.
\end{equation}
For real $f$,
\begin{equation}
 \|f-Q_{N,\beta}f\|_2^2\le\frac1{2\beta}\cE_{N,\beta}(f).
\end{equation}
\end{proposition}

\begin{proof}
Let $\omega=dz_1\wedge\cdots\wedge dz_N$ and replace the scalar datum $\bar\partial f$ by the $(N,1)$-form $\omega\wedge\bar\partial f$ in the trivial line bundle with weight $\Phi_{N,\beta}$.  It is $\bar\partial$-closed and square-integrable; compact support avoids boundary issues.  Demailly's weighted $\bar\partial$ estimate \cite[Chap.~VIII, Thm.~6.5]{DemaillyCADG} applies to $\Omega_N$ equipped with the standard Euclidean metric and complex structure, which furnish the required Kähler structure; \Cref{lem:levi} gives weak pseudoconvexity.  The smallest curvature eigenvalue is $\beta/2$, so the inverse-curvature constant is $2/\beta$.  The minimal solution is orthogonal to the weighted holomorphic kernel and therefore equals $f-Q_{N,\beta}f$.
For real $f$, $\sum_j|\bar\partial_jf|^2=|\nabla f|^2/4$.
\end{proof}

\begin{lemma}[Removal of collision poles]\label{lem:pole-removal}
If $a_N\le2$, every permutation-invariant $h\in\mathcal A^2_{N,\beta}$ extends to an entire function on $\C^N$.
\end{lemma}

\begin{proof}
Consider a point $z_*\in\{z_i=z_j\}$ such that no other coordinate pairs coincide.  Near such a point, use the normal coordinate $y=(z_i-z_j)/\sqrt2$ and tangential coordinates
$q=((z_i+z_j)/\sqrt2,(z_k)_{k\notin\{i,j\}})$.  Choose a product neighborhood $U\times B(0,R)$ meeting no other collision hyperplane.  By \eqref{eq:law}, the Gibbs density there is bounded below by a positive constant times $|y|^{a_N}$.  Fubini's theorem gives
\begin{equation}
 \int_{|y|<R}|y|^{a_N}|h(q,y)|^2\D y<\infty
 \quad\text{for almost every }q\in U.
\end{equation}
For such a $q$, set $f(y):=h(q,y)$ and $a:=a_N\in[0,2]$.  Apply the mean-value inequality for the subharmonic function $|f|^2$ on $B(y,|y|/2)$ \cite[Chap.~I, (4.8) and Thm.~4.12(c)]{DemaillyCADG}:
\begin{equation}\label{eq:pole-disk-mean}
 |f(y)|^2\le\frac4{\pi|y|^2}\int_{B(y,|y|/2)}|f(z)|^2\D z.
\end{equation}  Comparing $|z|^a$ with $|y|^a$ on that disk gives, for sufficiently small $|y|$,
\begin{equation}\label{eq:pole-shrinking-mass}
 \frac\pi{32}|y^2f(y)|^2
 \le4\int_{|z|<2|y|}|z|^a|f(z)|^2\D z
 \xrightarrow[y\to0]{}0.
\end{equation}
The last limit follows from absolute continuity of the integral.  Thus, $F(y):=yf(y)$ satisfies $F(y)=o(|y|^{-1})$ as $y\to0$, which is a removable singularity criterion for a holomorphic function on a punctured disk.  See \cite[Chap.~I, formula~(3.2)]{DemaillyCADG}.  Permutation invariance makes $f$ even, so the extension of $F$ is odd and vanishes at zero.  Dividing it by $y$ extends $f$ holomorphically across zero.

To obtain an extension holomorphic in both $q$ and $y$, fix $0<r<R$ and define
\begin{equation}\label{eq:pole-cauchy-extension}
 H(q,y):=\frac1{2\pi i}\int_{|\zeta|=r}
              \frac{h(q,\zeta)}{\zeta-y}\,\D\zeta,
 \qquad q\in U,\quad |y|<r.
\end{equation}
For $\zeta$ on the fixed circle, $h(q,\zeta)$ is holomorphic in $q\in U$, and $(\zeta-y)^{-1}$ is holomorphic for $|y|<r$.  Integration over this compact contour therefore makes $H$ jointly holomorphic in $(q,y)$.  For a.e. $q$, the one-variable extension and the Cauchy formula give $H(q,y)=h(q,y)$ for $0<|y|<r$.  Since this full-measure set of $q$ is dense in $U$, continuity gives the same identity for every $q\in U$.  Hence $h$ extends across every collision involving exactly one pair.

It remains to extend across the intersections of distinct collision hyperplanes, which have complex codimension at least two.  We use the local Cauchy construction for Hartogs extension; see \cite[Chap.~I, (6.3) and the proof of Cor.~6.4]{DemaillyCADG}.  Choose a complex direction transverse to each collision hyperplane, and write the hyperplanes as distinct linear graphs $y=\ell_\alpha(q)$.  At a fixed base point $q_0$, choose a circle $|y-y_0|=r$ surrounding the point under consideration such that $|\ell_\alpha(q_0)-y_0|\ne r$ for every $\alpha$.  By continuity and finiteness of the family, these inequalities persist for all $q$ in a sufficiently small neighborhood of $q_0$.  Define $H$ by the analogue of \eqref{eq:pole-cauchy-extension} on this circle.  For the dense set of $q$ for which the heights $\ell_\alpha(q)$ are pairwise distinct, the whole fiber disk lies in the domain already obtained by removing collisions involving exactly one pair.  The Cauchy formula gives $H=h$ on those disks.  Continuity then gives equality wherever $h$ was already defined.  The extensions agree on overlaps by this identity, and give an entire extension of $h$.
\end{proof}

\begin{lemma}[Common-phase bilinear identity]\label{lem:phase}
If $F,G$ are entire and $FG\in L^1(\mathbb P_{N,\beta})$, then
\begin{equation}
 \int_{\C^N} FG\,\D \mathbb P_{N,\beta}=F(0)G(0).
\end{equation}
\end{lemma}

\begin{proof}
The measure is invariant under $Z\mapsto e^{it}Z$.  Average the integral over $t\in[0,2\pi]$ and use the homogeneous Taylor expansions of $F$ and $G$; only total holomorphic degree zero survives.  Truncation justifies the averaging under the stated integrability.
\end{proof}

\begin{proposition}[Permutation-invariant Poincar\'e inequality]\label{prop:symmetric-pi}
If $a_N\le2$ and $f\in\mathsf D_{N,\beta}$ is real-valued and permutation-invariant, then
\begin{equation}\label{eq:sym-pi}
 \Var_{\mathbb P_{N,\beta}}(f)\le\frac1\beta\cE_{N,\beta}(f).
\end{equation}
\end{proposition}

\begin{proof}
First, let $f\in C_c^\infty(\Omega_N;\R)$ be permutation-invariant.  Write $Q=Q_{N,\beta}$, let $f_0=f-\int_{\C^N}f\,\D\mathbb P_{N,\beta}$, and set $h=Qf_0$, $u=f_0-h$.  Since $Q1=1$,
\begin{equation}
 u=f_0-Qf_0=f-Qf,
\end{equation}
so \Cref{prop:dbar} applies to the original compactly supported function $f$.  The permutation operators preserve both the measure and the holomorphic subspace, so they commute with $Q$.  Consequently, $h$ and $k:=Q\bar h$ are permutation-invariant and extend to entire functions by \Cref{lem:pole-removal}.  Since constants lie in $\mathcal A^2_{N,\beta}$, $h$ has mean zero; \Cref{lem:phase} with $G=1$ gives $h(0)=0$.

Projection orthogonality gives
\begin{equation}
 \|k\|_2^2=\int_{\C^N} \bar h\,\bar k\,\D \mathbb P_{N,\beta}
 =\overline{\int_{\C^N} hk\,\D \mathbb P_{N,\beta}}
 =\overline{h(0)k(0)}=0.
\end{equation}
Thus, $Q\bar h=0$.
Because $f_0$ is real,
\begin{equation}
 h=Qf_0=Q\bar f_0=Q\bar h+Q\bar u=Q\bar u,
\end{equation}
so $\|h\|_2\le\|u\|_2$.  Orthogonality and \Cref{prop:dbar} yield
\begin{equation}
 \Var_{\mathbb P_{N,\beta}}(f)=\|h\|_2^2+\|u\|_2^2
 \le2\|u\|_2^2\le\frac1\beta\cE_{N,\beta}(f).
\end{equation}

For a general real permutation-invariant $f\in\mathsf D_{N,\beta}$, choose real $g_m\in C_c^\infty(\Omega_N)$ converging to $f$ in the form norm and set $f_m=\mathsf Sg_m$.  The symmetrizer preserves the core and contracts both the $L^2$ norm and the Dirichlet energy.  Since $\mathsf Sf=f$, we have $f_m\to f$ in the form norm, and the result follows by passing to the limit.
\end{proof}

\begin{remark}[The Ginibre endpoint]\label{rem:ginibre-endpoint}

The endpoint $a_N=2$ in \Cref{lem:pole-removal,prop:symmetric-pi} has a direct Ginibre interpretation.  In the unit-droplet normalization, the complex Ginibre measure is
\begin{equation}
 \D\mathbb P_{N,2N}(Z)
 \propto e^{-N\sum_{i=1}^N|z_i|^2}|\Delta_N(Z)|^2\,\D Z,
\end{equation}
and the endpoint estimate reads
\begin{equation}
 \Var_{\mathbb P_{N,2N}}(f)
 \le \frac1{2N}\int_{\C^N}|\nabla f|^2\,\D\mathbb P_{N,2N}
\end{equation}
for every real permutation-invariant observable.  The real and imaginary parts of the center-of-mass observable show that the constant is sharp.  This is the sharp symmetric, equivalently unlabelled, Ginibre Poincar\'e inequality very recently obtained independently by Suzuki and Chafa\"i \cite{Suzuki2026Ginibre,Chafai2026Ginibre}. Our argument presented above does not use determinantal structure.  Suzuki further proved that the unlabelled spectral gap is strictly larger in the infinite-particle system than in every finite-particle system; the center-of-mass mode attaining the finite-particle gap has no $L^2$ counterpart in the infinite-particle Dirichlet-form domain \cite{Suzuki2026Ginibre}.  This comparison uses bulk coordinates $x=\sqrt N z$, in which the finite-particle Poincar\'e constant is $1/2$ for the full gradient energy (equivalently, the gap is $1$ for half that energy).

\end{remark}

\leavevmode The complex-analytic argument now covers every $N$ with $\beta/N\le2$.  For fixed $\beta$, only finitely many particle numbers satisfy $\beta/N>2$; the next lemma supplies a positive gap for these remaining cases.

\begin{lemma}[Fixed-particle compact resolvent]\label{lem:fixed-N}
For every fixed $(N,\beta)$, the nonnegative self-adjoint operator associated with the closed Dirichlet form $(\cE_{N,\beta},\mathsf D_{N,\beta})$ on $L^2(\mathbb P_{N,\beta})$ has compact resolvent, and its kernel consists only of constants.  Hence, $\mathbb P_{N,\beta}$ has a positive fixed-$(N,\beta)$ Poincar\'e gap.

Moreover, for every fixed $N$ and every compact interval $I\Subset(0,\infty)$,
\begin{equation}\label{eq:fixed-N-compact}
 \sup_{\beta\in I}C_P(\mathbb P_{N,\beta})<\infty.
\end{equation}

\end{lemma}

\begin{proof}
On $\Omega_N$, let $U:=\beta|Z|^2/2-a_N\log|\Delta_N|$ denote the negative logarithm of the unnormalized Gibbs density.  For $f\in C_c^\infty(\Omega_N)$, denote its ground-state transform by $\psi:=\mathsf{Z}_{N,\beta}^{-1/2}fe^{-U/2}$.  Writing $V_{N,\beta}$ for the resulting Schr\"odinger potential, the ground-state transform gives
\begin{equation}
 \cE_{N,\beta}(f)
 =\int_{\Omega_N}\left(|\nabla\psi|^2+V_{N,\beta}|\psi|^2\right)\D Z,
 \qquad V_{N,\beta}:=\frac14|\nabla U|^2-\frac12\Delta U.
\end{equation}
Away from collisions, Euler's identity applied to the homogeneity of the Vandermonde polynomial gives $Z\cdot\nabla\log|\Delta_N|=N(N-1)/2$.  Consequently,
\begin{equation}
 V_{N,\beta}(Z)\ge\frac{\beta^2}{4}|Z|^2-C_{N,\beta}.
\end{equation}

For $i<j$, set $y_{ij}:=(z_i-z_j)/\sqrt2$, the unit normal coordinate to the collision hyperplane $\{z_i=z_j\}$.  With $\Delta$ denoting the full real Laplacian on $\C^N$, the identity $\Delta_y\log|y|=2\pi\delta_0$ gives
\begin{equation}
 \Delta\log|z_i-z_j|=2\pi\delta_0(y_{ij}),
 \qquad
 \left(-\frac12\Delta U\right)_{\mathrm{sing}}
 =\pi a_N\sum_{1\le i<j\le N}\delta_0(y_{ij})\ge0.
\end{equation}
Here, $\delta_0(y_{ij})$ denotes Dirac mass in the normal coordinate, tensored with Lebesgue measure in the tangential directions.

After adding $C_{N,\beta}+1$ to the form, the preceding bound gives
\begin{equation}
 \int_{\Omega_N}\left(|\nabla\psi|^2+V_{N,\beta}|\psi|^2\right)\D Z
 +(C_{N,\beta}+1)\|\psi\|_2^2
 \ge c_{N,\beta}\int_{\Omega_N}
 \left(|\nabla\psi|^2+(1+|Z|^2)|\psi|^2\right)\D Z.
\end{equation}
For a family bounded in the core form norm, the ground-state-transformed functions therefore have uniformly bounded $L^2$ norms, gradients, and quadratic moments.  For every translation vector $h\in\C^N$ and $R>0$,
\begin{equation}
 \|\psi(\cdot+h)-\psi\|_2\le|h|\|\nabla\psi\|_2,
 \qquad \int_{|Z|>R}|\psi|^2\D Z
 \le R^{-2}\int|Z|^2|\psi|^2\D Z.
\end{equation}
Thus the family has uniform translation continuity and uniformly small tails.  The Fr\'echet--Kolmogorov criterion \cite[Cor.~4.27, p.~113]{Brezis2011} gives relative compactness in the common space $L^2(\C^N,\D Z)$.  Every element of the closed form ball is a form-norm limit of core elements in a slightly larger ball; the ground-state isometry therefore transfers this compactness to the full form domain.  Thus, the embedding into $L^2(\mathbb P_{N,\beta})$ is compact, and the associated operator has compact resolvent.  If the energy is zero, the weak gradient vanishes on the connected collision complement $\Omega_N$, so the function is constant almost everywhere.

Fix $N$ and write $I=[\beta_-,\beta_+]$.  Since $U=\beta\mathcal H_N$, the preceding core bounds are uniform for $\beta\in I$.  If \eqref{eq:fixed-N-compact} fails, core approximation and passage to a subsequence give $\beta_k\to\beta_\infty\in I$ and $f_k\in C_c^\infty(\Omega_N)$ such that
\begin{equation}
 \int_{\C^N} f_k\,\D\mathbb P_{N,\beta_k}\longrightarrow0,
 \qquad \int_{\C^N} f_k^2\,\D\mathbb P_{N,\beta_k}\longrightarrow1,
 \qquad \cE_{N,\beta_k}(f_k)\longrightarrow0.
\end{equation}
Set
\begin{equation}
 \phi_\beta:=\mathsf Z_{N,\beta}^{-1/2}e^{-\beta\mathcal H_N/2},
 \qquad \psi_k:=f_k\phi_{\beta_k}.
\end{equation}
The uniform core bounds and the same compactness criterion give $\psi_k\to\psi$ strongly in $L^2(\C^N)$ along a further subsequence.  Bounding the pair product by $e^{C_{N,I}|Z|}$ and completing the Gaussian square give a common majorant $C_{N,I}e^{-\beta_-|Z|^2/4}$ for $e^{-\beta\mathcal H_N}$ on $I$.  Dominated convergence gives
\begin{equation}\label{eq:fixed-N-normalizers}
 \mathsf Z_{N,\beta_k}=\int_{\C^N}e^{-\beta_k\mathcal H_N}\D Z
 \longrightarrow\int_{\C^N}e^{-\beta_\infty\mathcal H_N}\D Z
 =\mathsf Z_{N,\beta_\infty}>0,
 \qquad \|\phi_{\beta_k}-\phi_{\beta_\infty}\|_2\longrightarrow0.
\end{equation}
Hence,
\begin{equation}\label{eq:fixed-N-limit-orthogonality}
 \|\psi\|_2=1,
 \qquad \int_{\C^N}\psi\phi_{\beta_\infty}\,\D Z=0.
\end{equation}
For a smooth collision-free compactly supported $\psi_0$, the ground-state identity gives
\begin{align}
 \cE_{N,\beta}(\psi_0/\phi_\beta)
 &=\int|\nabla\psi_0|^2\D Z
   +\frac{\beta^2}{4}\int|\nabla\mathcal H_N|^2\psi_0^2\D Z
   -N\beta\int\psi_0^2\D Z.
\end{align}
In particular, $\int|\nabla\mathcal H_N|^2\psi_k^2\D Z$ is uniformly bounded, since $\beta_k\ge\beta_->0$ and the masses and energies are bounded.  Subtracting the identity at $\beta_k$ from that at $\beta_\infty$ gives
\begin{equation}\label{eq:fixed-N-energy-limit}
 \cE_{N,\beta_\infty}(\psi_k/\phi_{\beta_\infty})=\cE_{N,\beta_k}(f_k)
 +\frac{\beta_\infty^2-\beta_k^2}{4}\int|\nabla\mathcal H_N|^2\psi_k^2\D Z
 -N(\beta_\infty-\beta_k)\|\psi_k\|_2^2\xrightarrow[k\to\infty]{}0,
\end{equation}
where the first term tends to zero by construction, and the other two do so because their coefficients vanish while the drift integral and the mass remain bounded.
The ground-state isometry has closed range.  Thus, the strong $L^2$ limit of the transformed sequence corresponds to a limit in $L^2(\mathbb P_{N,\beta_\infty})$; closedness of the form and the vanishing gradients give an element $f\in\mathsf D_{N,\beta_\infty}$ of zero energy with transform $\psi$.  It is constant, and \eqref{eq:fixed-N-limit-orthogonality} forces that constant to vanish, contradicting $\|\psi\|_2=1$.  This proves \eqref{eq:fixed-N-compact}.
\end{proof}

Writing $C_P(\mathsf P)$ for the optimal Poincar\'e constant of a probability measure $\mathsf P$, define the uniform permutation-invariant bound $C_{\rm sym}(\beta)$ by
\begin{equation}\label{eq:C-sym}
 C_{\rm sym}(\beta):=
 \max\left\{\frac1\beta,
 \max_{2\le N\le\beta/2}C_P(\mathbb P_{N,\beta})\right\},
\end{equation}
with the finite inner maximum omitted when empty.  \Cref{prop:symmetric-pi} and \Cref{lem:fixed-N} give the permutation-invariant Poincar\'e bound with this constant for every $N$.

If $I=[\beta_-,\beta_+]\Subset(0,\infty)$, then \eqref{eq:fixed-N-compact} and the finiteness of the exceptional family give
\begin{equation}\label{eq:C-sym-compact}
 \sup_{\beta\in I}C_{\rm sym}(\beta)
 \le\max\left\{\frac1{\beta_-},
 \max_{2\le N\le\beta_+/2}\sup_{\beta\in I}C_P(\mathbb P_{N,\beta})\right\}<\infty,
\end{equation}
with the inner maximum again omitted when empty.

\subsection{Relative-coordinate conditionals and label-dependent fluctuations}

To estimate the change of $f$ under a transposition $\tau_{ij}$, fix $i<j$ and use
\begin{equation}
 c=\frac{z_i+z_j}{\sqrt2},
 \qquad u=u_{ij}:=\frac{z_i-z_j}{\sqrt2}.
\end{equation}
The transposition leaves $c$ fixed and sends $u$ to $-u$.  Write $Z_{-ij}:=(z_k)_{k\notin\{i,j\}}$ for the outside coordinates.  Condition on $(c,Z_{-ij})$, and set $A_k=c/\sqrt2-z_k$.  We first identify this conditional law, then use its evenness and its weighted Poincar\'e inequality to control the swap difference.

\begin{lemma}[Relative-coordinate conditional polynomial]\label{lem:relative-poly}
Let $\nu_{c,Z_{-ij}}$ denote the conditional law of $u$ given $(c,Z_{-ij})$, which is specified by
\begin{equation}\label{eq:relative-law}
 \D\nu_{c,Z_{-ij}}(u)
 \propto e^{-\beta|u|^2/2}|p(u)|^{\beta/N}\,\D u,
 \qquad
 p(u):=u\prod_{k\notin\{i,j\}}\left(A_k^2-\frac{u^2}{2}\right).
\end{equation}
The polynomial $p$ is odd, has degree $2N-3$, and has total exponent $\beta(2N-3)/N<2\beta$.
\end{lemma}

\begin{proof}
The pair change of variables is unitary.  Thus,
\begin{equation}
 |z_i|^2+|z_j|^2=|c|^2+|u|^2,
 \qquad \D z_i\,\D z_j=\D c\,\D u.
\end{equation}
After conditioning on $(c,Z_{-ij})$, all factors independent of $u$ are absorbed into the normalizing constant.  The internal pair contributes a constant times $|u|^{\beta/N}$, and for each $k\notin\{i,j\}$,
\begin{equation}
 |z_i-z_k||z_j-z_k|
 =\left|A_k^2-\frac{u^2}{2}\right|.
\end{equation}
This gives \eqref{eq:relative-law}.  Since $p$ is the product of the monomial $u$ and $N-2$ even quadratic factors, it follows that $p$ is odd of degree $2N-3$, and the weighted polynomial factor has total exponent $\beta(2N-3)/N<2\beta$.
\end{proof}

Define the relative-coordinate Poincar\'e constant by
\begin{equation}
 C_{\rm rel}(\beta):=C_{\rm wt}(\beta,2\beta).
\end{equation}

\leavevmode Writing $\nabla_i:=\nabla_{z_i}$, define the associated relative-coordinate gradient by $\nabla_{u_{ij}}:=(\nabla_i-\nabla_j)/\sqrt2$.

\begin{proposition}[Transposition estimate]\label{prop:swap}
For every pair $i<j$ and every real core function $f$,
\begin{equation}
 \|f-f\circ\tau_{ij}\|_2^2
 \le4C_{\rm rel}(\beta)
 \int_{\C^N}|\nabla_{u_{ij}}f|^2\,\D \mathbb P_{N,\beta}.
\end{equation}
\end{proposition}

\begin{proof}
Write $\xi=(c,Z_{-ij})$, let $\lambda$ be its marginal law, and let $f_\xi(u)$ be the original function $f$ expressed in these coordinates with $\xi$ fixed.  Set $m_\xi:=\int_\C f_\xi\D\nu_\xi$.  Since $p$ in \eqref{eq:relative-law} is odd, $\nu_\xi$ is even.  The triangle inequality and this symmetry give
\begin{align}\label{eq:conditional-swap-variance}
 \int_\C|f_\xi(u)-f_\xi(-u)|^2\,\D\nu_\xi(u)
 &\le2\int_\C|f_\xi(u)-m_\xi|^2\,\D\nu_\xi(u)
       +2\int_\C|f_\xi(-u)-m_\xi|^2\,\D\nu_\xi(u)\notag\\
 &=4\Var_{\nu_\xi}(f_\xi)\notag\\
 &\le4C_{\rm rel}(\beta)\int_\C|\nabla_u f_\xi|^2\,\D\nu_\xi.
\end{align}
Here the final inequality is an application of \Cref{thm:weighted}: the Gaussian coefficient is $\beta$ and the total exponent is $\beta(2N-3)/N<2\beta$.  The slice $f_\xi$ is a smooth compactly supported function avoiding the conditional collision points, so that theorem applies.  Finally, integrate \eqref{eq:conditional-swap-variance} against $\lambda$ and use
$\nabla_u f_\xi=(\nabla_i f-\nabla_j f)/\sqrt2$ to obtain
\begin{equation}
 \|f-f\circ\tau_{ij}\|_2^2
 \le4C_{\rm rel}(\beta)\int\!\int|\nabla_u f_\xi|^2\,\D\nu_\xi\,\D\lambda(\xi)
 =4C_{\rm rel}(\beta)\int_{\C^N}|\nabla_{u_{ij}}f|^2\,\D\mathbb P_{N,\beta}.
\end{equation}
\end{proof}

\leavevmode The next lemma supplies the two estimates needed to pass from a single transposition to all label-dependent fluctuations without losing uniformity in $N$: the random-transposition spectral-gap inequality \eqref{eq:RT} and the complete-graph gradient identity \eqref{eq:relative-gradient}.  Their $N^{-1}$ and $N$ factors cancel in \Cref{cor:nonsymmetric}.

\begin{lemma}[Random transpositions and complete-graph gradient identity]\label{lem:complete-graph}
For permutation averaging $\sfS$,
\begin{align}
 \|f-\sfS f\|_2^2&\le\frac1{2N}\sum_{i<j}\|f-f\circ\tau_{ij}\|_2^2,\label{eq:RT}\\
 \sum_{i<j}|\nabla_{u_{ij}}f|^2
 &=\frac12\sum_{i<j}|\nabla_if-\nabla_jf|^2
 \le\frac N2|\nabla f|^2.\label{eq:relative-gradient}
\end{align}
\end{lemma}

\begin{proof}
In the normalization of Diaconis and Shahshahani \cite[(1.1), Cor.~4, and Rem.~1]{DiaconisShahshahani1981}, the transition kernel assigns mass $1/N$ to the identity and $2/N^2$ to each transposition and has spectral gap $2/N$.  Its Poincar\'e inequality is exactly \eqref{eq:RT}.  The second identity follows from the definition of $\nabla_{u_{ij}}$ and the complete-graph variance identity.
\end{proof}

\begin{corollary}[Label-dependent fluctuations]\label{cor:nonsymmetric}
\begin{equation}
 \|f-\sfS f\|_2^2\le C_{\rm rel}(\beta)\cE_{N,\beta}(f).
\end{equation}
\end{corollary}

\begin{proof}
\Cref{prop:swap}, \eqref{eq:RT}, and \eqref{eq:relative-gradient} give
\begin{equation}
 \|f-\sfS f\|_2^2
 \le \frac{1}{2N}\sum_{i<j}
 4C_{\rm rel}(\beta)
 \int_{\C^N}|\nabla_{u_{ij}}f|^2\,\D \mathbb P_{N,\beta}
 \le \frac{2C_{\rm rel}(\beta)}{N}\cdot\frac N2\,
 \cE_{N,\beta}(f).
\end{equation}
The factor $N^{-1}$ from the random-transposition inequality cancels the factor $N$ from the complete-graph identity, leaving the stated constant.
\end{proof}

\begin{proof}[Proof of \Cref{thm:full-pi}]
For a real core function $f$, the orthogonal decomposition gives
\begin{equation}
 \Var(f)=\Var(\sfS f)+\|f-\sfS f\|_2^2.
\end{equation}
Use the permutation-invariant estimate, \Cref{cor:nonsymmetric}, and contraction of the Dirichlet form under averaging.
This proves \eqref{eq:full-pi} on the core.  The inequality extends to every real-valued $f\in\mathsf D_{N,\beta}$ by form-norm density.
\end{proof}


\section{A defective logarithmic Sobolev inequality}\label{sec:defective}

The goal of this section is to prove, uniformly in the particle number, the defective logarithmic Sobolev inequality stated in \Cref{cor:defective} for arbitrary labeled observables.  Unlike the variance argument of \Cref{sec:full-pi}, the entropy argument neither symmetrizes the density nor uses the orthogonal permutation decomposition.  It instead combines the weighted logarithmic Sobolev inequality \eqref{eq:weighted-lsi} applied to each one-site conditional, the conditional entropy estimate of \Cref{thm:entropy-factorization}, and the uniform repulsive partition bound of \Cref{thm:partition}.  The full Poincar\'e inequality of \Cref{thm:full-pi} re-enters only in \Cref{sec:tightening}, where the defective inequality is tightened.

\subsection{Modulation around thermal equilibrium}\label{sec:equilibrium}

To isolate the interaction term used in the entropy argument, we rewrite the interacting Gibbs measure relative to a product reference measure.  The appropriate reference is the thermal equilibrium measure: its Euler--Lagrange equation cancels the one-body terms and yields the exact canonical/modulated identity \eqref{eq:modulated}.  In that identity, the density relative to the product reference is a normalized exponential weight built from the modulated energy introduced in \eqref{eq:modulated-energy} below.  Its specialization to the thermal equilibrium measure gives the interaction in \eqref{eq:A-N}, whose exponential moments are controlled by \Cref{thm:partition}.  We therefore first construct the equilibrium measure and record the properties needed for this representation.

Let $\g(x)=-\log|x|$.  We minimize over probability measures.  For $\mu=\rho\D x$ with finite second moment and satisfying $\rho\log\rho\in L^1(\C)$ and $\g(x-y)\in L^1(\mu\otimes\mu)$, define the mean-field free energy $\cF_\beta$ by
\begin{equation}\label{eq:free-energy}
 \cF_\beta(\mu)
 :=\frac12\int_{\C}|x|^2\,\D\mu
 +\frac12\iint_{\C^2} \g(x-y)\,\D\mu(x)\D\mu(y)
 +\frac1\beta\int_{\C}\rho\log\rho\,\D x.
\end{equation}
Set $\cF_\beta=+\infty$ outside this finite-free-energy domain.

For probability measures $\mathsf Q\ll\mathsf P$ on a common measurable space $\mathsf X$, denote their relative entropy by
\begin{equation}
 \mathsf H(\mathsf Q\mid\mathsf P)
 :=\int_{\mathsf X}\log\!\left(\frac{\D\mathsf Q}{\D\mathsf P}\right)\,\D\mathsf Q,
\end{equation}
with the usual $+\infty$ convention otherwise.

\leavevmode The existence, uniqueness, and Euler--Lagrange characterization of the minimizer are standard in the theory of thermal equilibrium measures; compare \cite[Props.~2.16 and~2.19]{Serfaty2024Lectures}.  We include the proof to record the present normalization and to establish the regularity, tail, and compact-temperature bounds used below.

\begin{proposition}[Thermal equilibrium and compact-temperature bounds]\label{prop:equilibrium}
For every $\beta>0$, \eqref{eq:free-energy} has a unique minimizer $\mu_\beta=\rho_\beta\D x$, whose density $\rho_\beta$ is radial, strictly positive, bounded, and smooth.  Writing $U_\beta$ for the logarithmic potential generated by $\mu_\beta$, there exists a constant $C_\beta\in\R$ such that
\begin{equation}\label{eq:EL}
 \log\rho_\beta(x)
 =-\beta\left(\frac{|x|^2}{2}+U_\beta(x)\right)+C_\beta,
 \qquad U_\beta=\g*\mu_\beta.
\end{equation}
Moreover, for every multi-index $\alpha$ there exists a constant $0<C_{\beta,\alpha}'<\infty$ such that

\begin{equation}\label{eq:mu-tail}
 |\partial^\alpha\rho_\beta(x)|
 \le C_{\beta,\alpha}'(1+|x|)^{\beta+|\alpha|}e^{-\beta|x|^2/2}.
\end{equation}

Let
\begin{equation}
 c_\beta:=\iint_{\C^2} \g(x-y)\,\D\mu_\beta(x)\D\mu_\beta(y).
\end{equation}
For every compact interval $I=[\beta_-,\beta_+]\Subset(0,\infty)$,
\begin{equation}\label{eq:equilibrium-exact-bounds}
 \|\rho_\beta\|_{L^\infty(\C)}\le\frac1\pi,
 \qquad
 \int_{\C}|x|^2\,\D\mu_\beta(x)=\frac12+\frac2\beta
 \qquad (\beta\in I).
\end{equation}
Moreover, for every $m\ge0$ and $\theta>0$,
\begin{equation}\label{eq:equilibrium-compact-bounds}
 \sup_{\beta\in I}\left(
 |c_\beta|+
 \int_{\C}(1+|x|)^m\,\D\mu_\beta(x)+
 \int_{\C}e^{-\theta\beta U_\beta(x)}\,\D\mu_\beta(x)
 \right)<\infty.
\end{equation}

\end{proposition}

\begin{proof}
A Gaussian competitor has finite free energy.  The elementary bounds
\begin{equation}
 -\log|x-y|\ge-\log(1+|x|)-\log(1+|y|),
 \qquad \log(1+r)\le\delta r^2+C_\delta,
\end{equation}
and nonnegativity of relative entropy with respect to $\gamma_a(x)=(a/\pi)e^{-a|x|^2}$ imply, for suitable $a,\delta>0$, that there exists $K_\beta<\infty$ such that
\begin{equation}
 \cF_\beta(\mu)\ge\frac14\int_{\C}|x|^2\,\D\mu-K_\beta.
\end{equation}
Thus, a minimizing sequence is tight with bounded second moment.  For fixed $0<a<\beta/2$,
\begin{equation}
 \frac12\int_{\C}|x|^2\D\mu+\frac1\beta\int_{\C}\rho\log\rho
 =\frac1\beta \mathsf H(\mu\mid\gamma_a)
 +\left(\frac12-\frac a\beta\right)\int_{\C}|x|^2\D\mu
 -\frac1\beta\log(\pi/a),
\end{equation}
which is weakly lower semicontinuous.  The positive part of $\g$ is nonnegative and lower semicontinuous; its negative part is bounded by the two logarithmic moments and is uniformly integrable under the second-moment bound.  Truncation and Portmanteau therefore give lower semicontinuity of the interaction energy.  The direct method gives a minimizer.

For a smooth compactly supported signed density $\eta$ of mass zero, let $h=\g*\eta$.  Since $-\Delta h=2\pi\eta$ and the zero mass removes the boundary term at infinity, integration by parts yields
\begin{equation}
 \mathcal I(\eta):=\iint_{\C^2} \g(x-y)\eta(x)\eta(y)\,\D x\D y
 =\frac1{2\pi}\int_{\C}|\nabla h|^2\,\D x\ge0.
\end{equation}
If $\eta$ is bounded, compactly supported, and has mass zero, choose a nonnegative, compactly supported smooth approximate identity $\chi_\varepsilon$.  Then $\eta*\chi_\varepsilon$ has mass zero and
\begin{equation}
 \mathcal I(\eta*\chi_\varepsilon)\longrightarrow\mathcal I(\eta)
 \qquad(\varepsilon\downarrow0).
\end{equation}
Indeed, the mollifications are uniformly bounded and supported in a common compact set, where $|\log|x-y||$ is integrable; dominated convergence along an a.e.-convergent subsequence, followed by the subsequence criterion, proves the claim.  Thus $\mathcal I(\eta)\ge0$.  For bounded compactly supported probability densities $\rho,\sigma$ and $0\le t\le1$, all the following integrals are absolutely convergent, and
\begin{equation}
 \mathcal I((1-t)\rho+t\sigma)
 =(1-t)\mathcal I(\rho)+t\mathcal I(\sigma)
 -t(1-t)\mathcal I(\rho-\sigma)
 \le(1-t)\mathcal I(\rho)+t\mathcal I(\sigma).
\end{equation}
For a probability density $\rho$ in the finite-free-energy domain, set
\begin{equation}
 m_n:=\int_{|x|\le n}\min\{\rho(x),n\}\,\D x,
 \qquad
 \rho_n(x):=m_n^{-1}\mathbf1_{\{|x|\le n\}}\min\{\rho(x),n\}.
\end{equation}
Choose $n_0$ so that $m_n\ge1/2$ for $n\ge n_0$, and use only these indices in the definition of $\rho_n$.  Then $0\le\rho_n\le2\rho$.  Moreover, $\rho_n\to\rho$ a.e., $\rho_n\D x\rightharpoonup\rho\D x$, and
\begin{equation}
 \sup_{n\ge n_0}\int_{\C}|x|^2\rho_n(x)\,\D x<\infty,
 \qquad \mathcal I(\rho_n)\longrightarrow\mathcal I(\rho),
\end{equation}
where the energy convergence follows by domination by $4|\g(x-y)|\rho(x)\rho(y)$.  Construct $\sigma_n$ in the same way.  The mixtures $(1-t)\rho_n+t\sigma_n$ converge weakly to $(1-t)\rho+t\sigma$ and have uniformly bounded second moments.  The lower semicontinuity of the interaction energy established above therefore gives
\begin{equation}
 \mathcal I((1-t)\rho+t\sigma)
 \le\liminf_{n\to\infty}\mathcal I((1-t)\rho_n+t\sigma_n)
 \le(1-t)\mathcal I(\rho)+t\mathcal I(\sigma)<\infty.
\end{equation}
The negative part is integrable by the second-moment bound, so the mixture has absolutely integrable logarithmic interaction.  This proves convexity on the finite-free-energy domain.  The entropy is strictly convex, hence the minimizer is unique; rotational invariance implies radial symmetry of the minimizer.

Bounded multiplicative variations with zero mean under $\mu_\beta$ give the Euler--Lagrange equation \eqref{eq:EL} on the positive set of $\rho_\beta$.  To prove positivity, suppose the zero set has positive Lebesgue measure and choose a measurable bounded subset $A$ of it with $0<|A|<\infty$.  The auxiliary probability density $\psi=\mathbf1_A/|A|$ is bounded with finite second moment and finite entropy.  Its absolute logarithmic self-energy and its cross-energy with $\rho_\beta$ are finite: boundedness of $\psi$ controls the locally integrable logarithmic singularity, and the logarithmic moments control the negative part at infinity.  These facts use only the finite energy and second moment already obtained for the minimizer.  Consequently $\rho_t=(1-t)\rho_\beta+t\psi$, $0<t<1$, is an admissible finite-energy density, and its confinement and interaction changes are $O(t)$ as $t\downarrow0$, with $\beta$, $\rho_\beta$, and $A$ fixed.  The disjoint supports give the exact entropy identity
\begin{equation}
 \int\rho_t\log\rho_t
 =(1-t)\int\rho_\beta\log\rho_\beta+t\int\psi\log\psi
 +(1-t)\log(1-t)+t\log t.
\end{equation}
Thus $\cF_\beta(\rho_t\D x)-\cF_\beta(\mu_\beta)=\beta^{-1}t\log t+O(t)<0$ for sufficiently small $t$, contradicting minimality.  Hence $\rho_\beta>0$ a.e., and \eqref{eq:EL} holds Lebesgue-almost everywhere.  Since
\begin{equation}
 U_\beta(x)\ge-\log(1+|x|)-\int_{\C}\log(1+|y|)\,\D\mu_\beta(y),
\end{equation}
\eqref{eq:EL} gives the case $\alpha=0$ of \eqref{eq:mu-tail}; in particular $\rho_\beta\in L^\infty$.  Now, $-\Delta U_\beta=2\pi\rho_\beta$.  For this radial potential, Newton's formula gives a locally Lipschitz gradient, hence $U_\beta\in W^{2,p}_{\rm loc}$ for every finite $p$ and $C^{1,\theta}_{\rm loc}$ for every $0<\theta<1$.  The exponential form of \eqref{eq:EL} selects a strictly positive $C^1$ representative of the same density.  Writing the Newton gradient through the radial average $\int_0^1t\rho_\beta(tx)\D t$, differentiation of \eqref{eq:EL} inductively gives smoothness of both density and potential.

It remains to prove the derivative bounds.  Write $\rho_\beta(x)=q_\beta(r)$ and $U_\beta(x)=u_\beta(r)$, where $r=|x|$, and set
\begin{equation}
 M_\beta(r):=2\pi\int_0^r s q_\beta(s)\,\D s.
\end{equation}
Radiality and $-\Delta U_\beta=2\pi\rho_\beta$ give
\begin{equation}
 u_\beta'(r)=-\frac{M_\beta(r)}r,
 \qquad 0\le M_\beta(r)\le1.
\end{equation}
For $r\ge1$ and $k\ge2$, differentiating this identity shows that $u_\beta^{(k)}$ is bounded once $q_\beta^{(j)}$ is bounded for $0\le j\le k-2$; for $k=1$, boundedness follows directly from $0\le M_\beta\le1$.  Differentiating the exponential form of \eqref{eq:EL} and inducting on $k$ therefore gives
\begin{equation}
 |q_\beta^{(k)}(r)|
 \le C_{\beta,k}'(1+r)^{\beta+k}e^{-\beta r^2/2},
 \qquad r\ge1.
\end{equation}
Indeed, at the $k$th step, the derivatives of $r^2/2+u_\beta(r)$ of orders $2,\ldots,k$ are bounded, while its first derivative is $O(1+r)$.  For $|x|\ge1$, each Cartesian derivative of order $k$ of a radial function is a linear combination, with bounded angular coefficients, of terms $r^{j-k}q_\beta^{(j)}(r)$ with $1\le j\le k$.  Smoothness controls the unit ball, and \eqref{eq:mu-tail} follows.

We next establish the assertions uniform on compact inverse-temperature intervals.  Applying the Laplacian to \eqref{eq:EL} and using $-\Delta U_\beta=2\pi\rho_\beta$ gives
\begin{equation}
 \Delta\log\rho_\beta=2\beta(\pi\rho_\beta-1).
\end{equation}
The tail bound \eqref{eq:mu-tail} shows that $\rho_\beta$ attains its maximum.  At a maximum, the left-hand side is nonpositive, proving the first assertion in \eqref{eq:equilibrium-exact-bounds}.

For $r>0$, let $\mu_{\beta,r}$ be the pushforward of $\mu_\beta$ under $x\mapsto rx$.  The three terms of the free energy change by
\begin{equation}
 \frac12\int_{\C}|x|^2\,\D\mu_{\beta,r}
 =\frac{r^2}{2}\int_{\C}|x|^2\,\D\mu_\beta,
 \quad
 \frac12\iint_{\C^2}\g(x-y)\,\D\mu_{\beta,r}^{\otimes2}
 =\frac12c_\beta-\frac12\log r,
\end{equation}
and
\begin{equation}
 \int_{\C}\rho_{\beta,r}\log\rho_{\beta,r}
 =\int_{\C}\rho_\beta\log\rho_\beta-2\log r.
\end{equation}
Differentiating at the minimizing scale $r=1$ gives the second identity in \eqref{eq:equilibrium-exact-bounds}.

Consequently,
\begin{equation}
 L_I:=\sup_{\beta\in I}\int_{\C}\log(1+|y|)\,\D\mu_\beta(y)<\infty.
\end{equation}
The elementary logarithmic bound used in \eqref{eq:mu-tail} gives
\begin{equation}\label{eq:U-compact-lower}
 U_\beta(x)\ge-\log(1+|x|)-L_I.
\end{equation}
On the other hand, if $|x|\le1$, the density bound and the fact that $\g\le0$ on $\{|x-y|\ge1\}$ give
\begin{equation}
 U_\beta(x)
 \le\frac1\pi\int_{|h|<1}-\log|h|\,\D h=\frac12.
\end{equation}
Since \eqref{eq:EL} and normalization imply
\begin{equation}
 e^{C_\beta}
 =\left(\int_{\C}e^{-\beta(|x|^2/2+U_\beta(x))}\,\D x\right)^{-1},
\end{equation}
integration over the unit disk yields $e^{C_\beta}\le\pi^{-1}e^{\beta_+}$.  Combining this with \eqref{eq:U-compact-lower} gives the uniform tail
\begin{equation}\label{eq:mu-compact-tail}
 \rho_\beta(x)
 \le C_I e^{-\beta_-|x|^2/2}(1+|x|)^{\beta_+},
 \qquad \beta\in I.
\end{equation}
This proves the uniform polynomial-moment bound.  Furthermore,
\begin{equation}
 \iint_{\C^2}\g^+(x-y)\,\D\mu_\beta(x)\D\mu_\beta(y)\le\frac12,
 \qquad
 \g^-(x-y)\le\log(1+|x|)+\log(1+|y|),
\end{equation}
so $|c_\beta|$ is uniformly bounded.  Finally, \eqref{eq:U-compact-lower} gives
\begin{equation}
 e^{-\theta\beta U_\beta(x)}
 \le e^{\theta\beta_+L_I}(1+|x|)^{\theta\beta_+},
\end{equation}
and the last assertion of \eqref{eq:equilibrium-compact-bounds} follows from \eqref{eq:mu-compact-tail}.

\end{proof}

\leavevmode Having constructed the product reference, we now introduce the modulated energy and its associated partition function and Gibbs measure, specialize them to thermal equilibrium, and establish the exact identity with the original Gibbs measure.

For $X_N=(x_1,\ldots,x_N)$, denote its empirical measure by
\begin{equation}
 \mu_N:=\frac1N\sum_{i=1}^N\delta_{x_i}.
\end{equation}
For a reference probability measure $\mu$, let $\Delta:=\{(x,x):x\in\C\}$ denote the diagonal and define the modulated energy $\mathsf F_N(X_N,\mu)$ by
\begin{equation}\label{eq:modulated-energy}
 \mathsf F_N(X_N,\mu)
 :=\frac12\iint_{(\C^2)\setminus\Delta}
 \g(x-y)\,\D(\mu_N-\mu)(x)\D(\mu_N-\mu)(y),
\end{equation}
whenever the following absolute-integrability condition holds.  With $\nu=\mu_N-\mu$ and $\Lambda=(\nu\otimes\nu)|_{\C^2\setminus\Delta}$, we require
\begin{equation}
 \int_{\C^2\setminus\Delta}|\g(x-y)|\,\D|\Lambda|(x,y)<\infty,
\end{equation}
where $|\Lambda|$ is the total variation measure of the signed measure $\Lambda$.  Only $\mu=\mu_\beta$ is used below.  For $t>0$, define the modulated partition function $\mathsf K_{N,t}(\mu)$ and, whenever $0<\mathsf K_{N,t}(\mu)<\infty$, the modulated Gibbs measure $\mathbb Q_{N,t}(\mu)$ by
\begin{align}
 \mathsf K_{N,t}(\mu)
 &:=\int_{\C^N} e^{-tN\mathsf F_N(X_N,\mu)}\,\D\mu^{\otimes N},\label{eq:modulated-partition}\\
 \D\mathbb Q_{N,t}(\mu)(X_N)
 &:=\mathsf K_{N,t}(\mu)^{-1}
 e^{-tN\mathsf F_N(X_N,\mu)}\,\D\mu^{\otimes N}(X_N).
 \label{eq:modulated-gibbs}
\end{align}
The normalizing factor $\mathsf K_{N,t}(\mu)^{-1}$ makes $\mathbb Q_{N,t}(\mu)$ a probability measure.

For $X_N\in\Omega_N$, expanding \eqref{eq:modulated-energy} at $\mu=\mu_\beta$ gives
\begin{equation}\label{eq:A-N}
 A_{N,\beta}(X_N)
 :=N\mathsf F_N(X_N,\mu_\beta)
 =\frac1{2N}\sum_{i\ne j}\g(x_i-x_j)
 -\sum_{i=1}^NU_\beta(x_i)+\frac N2c_\beta.
\end{equation}
The bounded density and finite second moment of $\mu_\beta$ give
\begin{equation}
 \int_{\C}|\g(x-y)|\,\D\mu_\beta(y)<\infty\quad(x\in\C),
 \qquad
 \iint_{\C^2}|\g(x-y)|\,\D\mu_\beta(x)\D\mu_\beta(y)<\infty.
\end{equation}
These bounds justify the expansion in \eqref{eq:A-N}.  Since $\mu_\beta$ has a density, the collision set has $\mu_\beta^{\otimes N}$-measure zero.  Thus subsequent density identities hold $\mu_\beta^{\otimes N}$-a.e.; assigning arbitrary values to the densities on that null set leaves their integrals, and hence the resulting measures, unchanged.

\leavevmode The next identity is the thermal-equilibrium splitting in the terminology of \cite[Lemma~5.2 and (5.1.9)--(5.1.14)]{Serfaty2024Lectures}; the proof records the present normalization directly.

\begin{proposition}[Canonical/modulated identity]\label{prop:modulated}
One has the exact thermal-equilibrium splitting
\begin{equation}\label{eq:modulated}
 \mathbb P_{N,\beta}=\mathbb Q_{N,\beta}(\mu_\beta),
 \qquad
 \D\mathbb P_{N,\beta}
 =\mathsf K_{N,\beta}(\mu_\beta)^{-1}e^{-\beta A_{N,\beta}}\,\D\mu_\beta^{\otimes N}.
\end{equation}
\end{proposition}

\begin{proof}
Insert \eqref{eq:EL} into $\mu_\beta^{\otimes N}$ and multiply by $e^{-\beta A_{N,\beta}}$.  The terms $\sum_iU_\beta(x_i)$ cancel, giving
\begin{equation}
 e^{-\beta A_{N,\beta}}\prod_{i=1}^N\rho_\beta(x_i)
 =e^{NC_\beta-\beta Nc_\beta/2}
 \exp\left\{-\frac\beta2\sum_i|x_i|^2
 -\frac\beta N\sum_{i<j}\g(x_i-x_j)\right\}.
\end{equation}
By \Cref{prop:wellposed}, integration yields
$\mathsf K_{N,\beta}(\mu_\beta)=e^{NC_\beta-\beta Nc_\beta/2}\mathsf Z_{N,\beta}\in(0,\infty)$.
Dividing by this normalizer proves \eqref{eq:modulated}.
\end{proof}


\subsection{The repulsive partition estimate}\label{sec:partition}

The entropy argument uses the following uniform bound for the partition function defined in \eqref{eq:modulated-partition}.  It follows from the upper-bound argument of \cite[Theorem~2.5 and Appendix~A]{DelgadinoGvalani2025}, with the centered fluctuation and the regularization-error identity specified in Section~2 of the accompanying supplement.  In our normalization, for $X_N\in\Omega_N$,
\begin{equation}\label{eq:partition-centered-normalization}
 \eta_N=\sqrt N(\mu_N-\mu_\beta),\qquad
 \eta_N(\C)=0,\qquad
 \frac12\iint_{\C^2\setminus\Delta}\g(x-y)\,\D\eta_N(x)\D\eta_N(y)
 =A_{N,\beta}=N\mathsf F_N(X_N,\mu_\beta).
\end{equation}

\begin{proposition}[Repulsive logarithmic partition bound]\label{thm:partition}
For every compact interval $I\Subset(0,\infty)$ and every $T>0$,
\begin{equation}\label{eq:partition}
 \sup_{\beta\in I}\sup_{0<t\le T}\sup_{N\ge2}
 \mathsf K_{N,t}(\mu_\beta)<\infty.
\end{equation}
\end{proposition}

Jensen's inequality also gives the uniform lower bound
\begin{equation}\label{eq:K-lower}
 \mathsf K_{N,t}(\mu_\beta)
 \ge\exp\{-t\E_{\mu_\beta^{\otimes N}}[A_{N,\beta}]\}
 =\exp(tc_\beta/2),
 \qquad t>0.
\end{equation}
By \eqref{eq:equilibrium-compact-bounds}, this lower bound is uniform for $\beta$ in a compact interval and $t$ in a bounded interval.

\subsection{Conditional entropy and the additive defect}\label{sec:entropy}

The conditional estimate below is formulated for an arbitrary labeled density.  Its proof uses the entropy chain rule rather than the orthogonal decomposition from \Cref{sec:full-pi}.

For $i\in[N]$, denote by $X_{-i}$ the vector of outside coordinates,
\begin{equation}
 X_{-i}:=(x_1,\ldots,x_{i-1},x_{i+1},\ldots,x_N).
\end{equation}
For $\mathsf Q\ll\mathbb P_{N,\beta}$, denote by $\mathsf Q_i(\cdot\mid X_{-i})$ and $\mathbb P_i(\cdot\mid X_{-i})$ the one-site conditional laws under $\mathsf Q$ and $\mathbb P_{N,\beta}$, respectively, and by $\mathsf Q_{-i}$ the $X_{-i}$-marginal of $\mathsf Q$.  Denote the summed one-site conditional relative entropy by

\begin{equation}
 \cS_N(\mathsf Q)
 :=\sum_{i=1}^N\E_{\mathsf Q}\Bigl[
 \mathsf H\bigl(\mathsf Q_i(\cdot\mid X_{-i})\mid \mathbb P_i(\cdot\mid X_{-i})\bigr)\Bigr].
\end{equation}

\begin{proposition}[Conditional entropy inequality with an additive defect]\label{thm:entropy-factorization}
Assume $\mathsf Q$ is a probability measure such that $\mathsf H(\mathsf Q\mid\mathbb P_{N,\beta})<\infty$ and
$A_{N,\beta}+N^{-1}\sum_{i=1}^NU_\beta(x_i)\in L^1(\mathsf Q)$.  Then, for every $\lambda>1$,
\begin{align}
 \cS_N(\mathsf Q)
 &\ge\left(1-\frac1\lambda\right)\mathsf H(\mathsf Q\mid\mathbb P_{N,\beta})
 -\log \mathsf K_{N,\beta}(\mu_\beta)\notag\\
 &\quad-\frac1\lambda\log\E_{\mathbb P_{N,\beta}}\left[
 \exp\left\{-\lambda\beta\left(A_{N,\beta}
 +\frac1N\sum_{i=1}^NU_\beta(x_i)\right)\right\}\right].\label{eq:entropy-factorization}
\end{align}
Moreover,
\begin{equation}\label{eq:D-uniform}
 \mathfrak D_{\beta,\lambda}:=\sup_{M\ge2}\left\{\log \mathsf K_{M,\beta}(\mu_\beta)+\lambda^{-1}\log\E_{\mathbb P_{M,\beta}}\left[e^{-\lambda\beta\left(A_{M,\beta}+M^{-1}\sum_{i=1}^M U_\beta(x_i)\right)}\right]\right\}<\infty.
\end{equation}
In particular, these hypotheses cover $\D\mathsf Q=f^2\,\D\mathbb P_{N,\beta}$ for every bounded collision-free compactly supported smooth $f$ with $\int_{\C^N} f^2\,\D\mathbb P_{N,\beta}=1$.
\end{proposition}

\begin{proof}
We work under the displayed finite-entropy and integrability hypotheses.  Set $a=\beta/8$.  Then \eqref{eq:gibbs-polynomial-bound} gives $\E_{\mathbb P_{N,\beta}}[e^{a\sum_i|x_i|^2}]<\infty$.  The Donsker--Varadhan lemma gives
\begin{equation}
 a\E_{\mathsf Q}\Bigl[\sum_i|x_i|^2\Bigr]
 \le\mathsf H(\mathsf Q\mid\mathbb P_{N,\beta})
 +\log\E_{\mathbb P_{N,\beta}}[e^{a\sum_i|x_i|^2}]<\infty.
\end{equation}
Thus $\mathsf Q$ has finite quadratic and logarithmic moments.  The bounds $-\log(1+|x|)-C\le U_\beta(x)\le1/2$ then give $U_\beta(x_i)\in L^1(\mathsf Q)$ for each $i$, and the assumed integrability of $A_{N,\beta}+N^{-1}\sum_iU_\beta(x_i)$ gives $A_{N,\beta}\in L^1(\mathsf Q)$.  By \eqref{eq:modulated},
\begin{equation}\label{eq:canonical-entropy-identity}
 \mathsf H(\mathsf Q\mid\mu_\beta^{\otimes N})
 =\mathsf H(\mathsf Q\mid\mathbb P_{N,\beta})
 -\beta\E_{\mathsf Q}[A_{N,\beta}]
 -\log\mathsf K_{N,\beta}(\mu_\beta)<\infty.
\end{equation}
The entropy chain rule now shows that, for each $i$, both
$\mathsf H(\mathsf Q_{-i}\mid\mu_\beta^{\otimes(N-1)})$ and
$\E_{\mathsf Q}[\mathsf H(\mathsf Q_i(\cdot\mid X_{-i})\mid\mu_\beta)]$ are finite: they are nonnegative and sum to \eqref{eq:canonical-entropy-identity}.  The conditional law $\mathbb P_i(\cdot\mid X_{-i})$ of the $i$th particle under $\mathbb P_{N,\beta}$ is proportional, relative to $\mu_\beta$, to
\begin{equation}
 e^{-\beta\Phi_i(z;X_{-i})},
 \qquad
 \Phi_i(z;X_{-i})=\frac1N\sum_{j\ne i}\g(z-x_j)-U_\beta(z).
\end{equation}
The following equality is the entropy chain rule, and the final bound is the relative-entropy Han--Shearer inequality; see Han \cite{Han1978} and, for the Polish product-space formulation used here, Madiman and Tetali \cite[Thm.~V, Cor.~VII, and (21)--(22)]{MadimanTetali2010}:
\begin{align}\label{eq:han-shearer}
 &\sum_{i=1}^N\E_{\mathsf Q}\Bigl[
 \mathsf H\bigl(\mathsf Q_i(\cdot\mid X_{-i})\mid\mu_\beta\bigr)\Bigr]\notag\\
 &\qquad=N \mathsf H(\mathsf Q\mid\mu_\beta^{\otimes N})
 -\sum_{i=1}^N\mathsf H(\mathsf Q_{-i}\mid\mu_\beta^{\otimes(N-1)})
 \ge \mathsf H(\mathsf Q\mid\mu_\beta^{\otimes N}).
\end{align}
Expanding each conditional entropy relative to its tilted law and using Jensen's lower bound on its normalizing constant give
\begin{equation}
 \cS_N(\mathsf Q)
 \ge \mathsf H(\mathsf Q\mid\mu_\beta^{\otimes N})
 +\beta\E_{\mathsf Q}\left[\sum_i\left(\Phi_i(x_i)-\int_{\C}\Phi_i\,\D\mu_\beta\right)\right].
\end{equation}
Algebraic manipulation yields
\begin{equation}
 \sum_i\left(\Phi_i(x_i)-\int_{\C}\Phi_i\D\mu_\beta\right)
 =2A_{N,\beta}(X)+\frac1N\sum_iU_\beta(x_i)
 =:A_{N,\beta}(X)+B_{N,\beta}(X).
\end{equation}
Combining this identity with \eqref{eq:canonical-entropy-identity}, we obtain
\begin{equation}
 \cS_N(\mathsf Q)\ge \mathsf H(\mathsf Q\mid\mathbb P_{N,\beta})+\beta\E_{\mathsf Q}[B_{N,\beta}]-\log \mathsf K_{N,\beta}(\mu_\beta).
\end{equation}
Donsker--Varadhan gives
\begin{equation}
 \beta\E_{\mathsf Q}[B_{N,\beta}]
 \ge-\frac1\lambda \mathsf H(\mathsf Q\mid\mathbb P_{N,\beta})
 -\frac1\lambda\log\E_{\mathbb P_{N,\beta}}\bigl[e^{-\lambda\beta B_{N,\beta}}\bigr],
\end{equation}
proving \eqref{eq:entropy-factorization}.

By \eqref{eq:modulated},
\begin{equation}
 \E_{\mathbb P_{N,\beta}}[e^{-\lambda\beta B_{N,\beta}}]
 =\mathsf K_{N,\beta}(\mu_\beta)^{-1}
 \E_{\mu_\beta^{\otimes N}}\left[
 e^{-(1+\lambda)\beta A_{N,\beta}}
 e^{-(\lambda\beta/N)\sum_iU_\beta(x_i)}\right].
\end{equation}
Cauchy--Schwarz then gives
\begin{equation}\label{eq:defect-moment}
 \E_{\mathbb P_{N,\beta}}\bigl[e^{-\lambda\beta B_{N,\beta}}\bigr]
 \le \mathsf K_{N,\beta}(\mu_\beta)^{-1}
 \mathsf K_{N,2(1+\lambda)\beta}(\mu_\beta)^{1/2}
 \left(\int_{\C} e^{-(2\lambda\beta/N)U_\beta}\,\D\mu_\beta\right)^{N/2}.
\end{equation}
For $X=e^{-2\lambda\beta U_\beta}$, concavity of $x^{1/N}$ implies
\begin{equation}
 (\E[X^{1/N}])^N\le\E[X].
\end{equation}
The one-body moment $\int_{\C} e^{-2\lambda\beta U_\beta}\D\mu_\beta$ is finite because $U_\beta(x)\ge-\log(1+|x|)-C$ and $\mu_\beta$ has Gaussian-polynomial tails.  \Cref{thm:partition} bounds both $\mathsf K_{N,\beta}(\mu_\beta)$ and $\mathsf K_{N,2(1+\lambda)\beta}(\mu_\beta)$ uniformly in $N$, while \eqref{eq:K-lower} gives $\mathsf K_{N,\beta}(\mu_\beta)^{-1}\le e^{-\beta c_\beta/2}$.  This proves \eqref{eq:D-uniform} and shows that $e^{-\lambda\beta B_{N,\beta}}\in L^1(\mathbb P_{N,\beta})$, hence $B_{N,\beta}^-\in L^1(\mathbb P_{N,\beta})$.

To control the positive part, \eqref{eq:A-N} gives
\begin{equation}
 B_{N,\beta}
 =\frac1N\sum_{i<j}\g(x_i-x_j)
 -\left(1-\frac1N\right)\sum_{i=1}^NU_\beta(x_i)
 +\frac N2c_\beta.
\end{equation}
The lower bound on $U_\beta$ used above gives $(-U_\beta(x))^+\le\log(1+|x|)+C$, and these one-body terms are integrable under $\mathbb P_{N,\beta}$ by \eqref{eq:gibbs-polynomial-bound} and the Gaussian factor in \eqref{eq:law}.  For each pair, with $r=|x_i-x_j|$,
\begin{equation}
 r^{\beta/N}\g^+(x_i-x_j)=r^{\beta/N}(-\log r)_+
\end{equation}
is bounded.  Applying \eqref{eq:gibbs-polynomial-bound} to the remaining pair factors therefore gives $\g^+(x_i-x_j)\in L^1(\mathbb P_{N,\beta})$.  Since $c_\beta$ is finite, it follows that $B_{N,\beta}^+\in L^1(\mathbb P_{N,\beta})$, and hence $B_{N,\beta}\in L^1(\mathbb P_{N,\beta})$.

We may therefore take $\mathsf Q=\mathbb P_{N,\beta}$ in \eqref{eq:entropy-factorization}; since both entropies then vanish,
\begin{equation}
 \log \mathsf K_{N,\beta}(\mu_\beta)
 +\frac1\lambda\log\E_{\mathbb P_{N,\beta}}\bigl[e^{-\lambda\beta B_{N,\beta}}\bigr]
 \ge0.
\end{equation}

At $\lambda=2$, the partition parameters are $\beta$ and $6\beta$, while the one-body exponent is $4\beta$.  Combining \eqref{eq:defect-moment} with the concavity estimate gives
\begin{equation}\label{eq:D-compact-bound}
 \log \mathsf K_{N,\beta}(\mu_\beta)
 +\frac12\log\E_{\mathbb P_{N,\beta}}\bigl[e^{-2\beta B_{N,\beta}}\bigr]
 \le\frac12\log\mathsf K_{N,\beta}(\mu_\beta)
 +\frac14\log\mathsf K_{N,6\beta}(\mu_\beta)
 +\frac14\log\int_{\C} e^{-4\beta U_\beta}\,\D\mu_\beta.
\end{equation}
Therefore, \Cref{prop:equilibrium,thm:partition} imply that, for every compact interval $I\Subset(0,\infty)$,
\begin{equation}\label{eq:D-compact}
 \sup_{\beta\in I}\mathfrak D_{\beta,2}<\infty.
\end{equation}

\end{proof}

\leavevmode At $\lambda=2$, the conditional-entropy lower bound \eqref{eq:entropy-factorization} combines with the uniform one-site application of the weighted logarithmic Sobolev inequality \eqref{eq:weighted-lsi} from \Cref{thm:weighted}, with $\gamma=\beta$ and $Q=\beta$, to give \Cref{cor:defective}.  This is the defective inequality used in the Rothaus tightening of \Cref{sec:tightening}.

Denote by $\rho_{\rm one}(\beta)$ the uniform one-site logarithmic Sobolev constant and by $\mathfrak D_\beta$ the uniform defect at $\lambda=2$:
\begin{equation}
 \rho_{\rm one}(\beta):=\rho_{\rm wt}(\beta,\beta),
 \qquad \mathfrak D_\beta:=\mathfrak D_{\beta,2}.
\end{equation}

\begin{corollary}[Uniform defective LSI]\label{cor:defective}
For every fixed $\beta>0$, every $N\ge2$, and every real-valued $f\in \mathsf D_{N,\beta}$,
\begin{equation}\label{eq:defective}
 \Ent_{\mathbb P_{N,\beta}}(f^2)
 \le\frac4{\rho_{\rm one}(\beta)}\cE_{N,\beta}(f)
 +2\mathfrak D_\beta\int_{\C^N} f^2\,\D \mathbb P_{N,\beta}.
\end{equation}
\end{corollary}

\begin{proof}
For a normalized bounded core function, set $\D\mathsf Q=f^2\,\D\mathbb P_{N,\beta}$. The one-site conditional law is
\begin{equation}
 \mathbb P_i(dz\mid X_{-i})\propto e^{-\beta|z|^2/2}\prod_{j\ne i}|z-x_j|^{\beta/N}\,\D z,
\end{equation}
whose total exponent is less than $\beta$.  \Cref{thm:weighted} gives the conditional LSI uniformly in the conditioned configuration.  The conditional chain rule gives
\begin{equation}
 \cS_N(\mathsf Q)\le\frac2{\rho_{\rm one}(\beta)}\cE_{N,\beta}(f).
\end{equation}
Apply \Cref{thm:entropy-factorization} with $\lambda=2$ and rescale by homogeneity.

For general $f\in \mathsf D_{N,\beta}$, approximate in form norm by core functions.  Entropy is lower semicontinuous under $L^2$ convergence (after passing to a subsequence and using the standard lower bound for $s\log s$), while the energy and $L^2$ terms converge.  Therefore, \eqref{eq:defective} holds on the closed domain before any centering operation is performed.
\end{proof}


\section{Proof of the main theorem}\label{sec:tightening}

\begin{proof}[Proof of \Cref{thm:main}]
First take $f\in C_c^\infty(\Omega_N)$ and apply the Rothaus centering inequality \eqref{eq:rothaus-centering} with $\mathsf P=\mathbb P_{N,\beta}$.  By \Cref{lem:constants-domain}, $f-\E_{\mathbb P_{N,\beta}}[f]\in \mathsf D_{N,\beta}$.  The defective inequality \eqref{eq:defective} holds on the full closed domain $\mathsf D_{N,\beta}$ by \Cref{cor:defective}.  Apply it to $f-\E_{\mathbb P_{N,\beta}}[f]$, and then apply \Cref{thm:full-pi} to obtain

\begin{equation}\label{eq:final-lsi-coefficient}
 \Ent_{\mathbb P_{N,\beta}}(f^2)
 \le\left[
 \frac4{\rho_{\rm one}(\beta)}
 +(2\mathfrak D_\beta+2)C_P^{\rm all}(\beta)
 \right]\cE_{N,\beta}(f).
\end{equation}
Form approximation and lower semicontinuity of entropy extend \eqref{eq:final-lsi-coefficient} to every $f\in\mathsf D_{N,\beta}$.  Since $C_P^{\rm all}(\beta)=C_{\rm sym}(\beta)+C_{\rm rel}(\beta)$, an explicit $N$-uniform lower bound for the logarithmic Sobolev constants is
\begin{equation}\label{eq:rho-star}
 \rho_*(\beta)
 :=\frac{2}{
 4/\rho_{\rm one}(\beta)
 +(2\mathfrak D_\beta+2)
  \bigl(C_{\rm sym}(\beta)+C_{\rm rel}(\beta)\bigr)}.
\end{equation}
Every constituent is finite for fixed $\beta$ and independent of $N$, which proves the theorem.

It remains to verify the compact-temperature assertion.  Fix $I=[\beta_-,\beta_+]\Subset(0,\infty)$ and set
\begin{equation}
 \rho_{{\rm one},I}:=\beta_-\rho_{\rm wt}(1,\beta_+),
 \qquad
 C_{{\rm rel},I}:=\beta_-^{-1}C_{\rm wt}(1,2\beta_+),
\end{equation}
\begin{equation}
 C_{{\rm sym},I}:=\sup_{\beta\in I}C_{\rm sym}(\beta),
 \qquad
 \mathfrak D_I:=\sup_{\beta\in I}\mathfrak D_\beta.
\end{equation}
The first two constants are finite and positive by \eqref{eq:weighted-scaling}; they are common admissible constants because the one-site and relative-coordinate total exponents are bounded by $\beta_+$ and $2\beta_+$, respectively.  The other two constants are finite by \eqref{eq:C-sym-compact} and \eqref{eq:D-compact}.  Repeating the defective-LSI and full Poincar\'e arguments with these common constants gives, simultaneously for $\beta\in I$ and $N\ge2$,
\begin{equation}
 \Ent_{\mathbb P_{N,\beta}}(f^2)
 \le\left[
 \frac4{\rho_{{\rm one},I}}
 +(2\mathfrak D_I+2)
  \bigl(C_{{\rm sym},I}+C_{{\rm rel},I}\bigr)
 \right]\cE_{N,\beta}(f).
\end{equation}
Thus, the compact-temperature assertion holds with
\begin{equation}\label{eq:rho-I}
 \rho_I
 :=\frac{2}{
 4/\rho_{{\rm one},I}
 +(2\mathfrak D_I+2)
  \bigl(C_{{\rm sym},I}+C_{{\rm rel},I}\bigr)}>0.
\end{equation}

\end{proof}

\begin{remark}[The explicit lower bound]\label{rem:explicit-rate}
The lower bound in \eqref{eq:rho-star} records the constants produced by the weighted one-particle, permutation-invariant, relative-coordinate, partition, and Rothaus steps.  The same constituents, chosen uniformly on a compact inverse-temperature interval, give \eqref{eq:rho-I}.  Neither lower bound is claimed to be sharp, and the argument does not provide one positive lower bound uniform as $\beta\to0$ or $\beta\to\infty$.
\end{remark}


\section{Perspective and outlook}\label{sec:discussion}

We close by discussing the dependence on inverse temperature, the points at which the planar logarithmic structure enters, and the consequences and extensions suggested by the theorem.  The final subsections place our result within the author's broader Coulomb/Riesz program.

\subsection{Temperature dependence and effectivity}

The estimates are locally uniform in inverse temperature: for every compact interval $I\Subset(0,\infty)$, the lower bound in \Cref{thm:main} may be chosen uniformly for $\beta\in I$.  The point-weight constants in \Cref{thm:weighted}, the finite exceptional family entering $C_{\rm sym}(\beta)$, the partition defect $\mathfrak D_\beta$, and the final Rothaus tightening step all admit common finite bounds on such intervals.  Some of this dependence is explicit, while the quantitative strong-$A_\infty$ and weighted Sobolev constants and the fixed-particle spectral gaps are not presently effective enough for sharp temperature asymptotics.  The argument therefore does not provide one lower bound uniform on all of $(0,\infty)$ or controlled asymptotics as $\beta\to0$ or $\beta\to\infty$.

The theorem holds for every fixed inverse temperature $\beta>0$, rather than only for sufficiently small $\beta$, because the weighted one-particle estimate remains valid at every finite total exponent.  Indeed, the one-site conditionals have total exponent below $\beta$, and the relative-coordinate conditionals have total exponent below $2\beta$.  Thus, neither conditional step imposes an upper bound on $\beta$.  Small-$\beta$ restrictions arising from Hessian absorption or Dobrushin-type influence estimates reflect limitations of those criteria, not of the LSI established here.

\subsection{Where the planar logarithmic structure enters}

Four features of the argument are specific to the present model.  First, products of positive powers of planar distances have a conformal-metric interpretation that gives quantitative strong-$A_\infty$ geometry.  Second, the logarithmic force field satisfies the critical weak-$L^2$ estimate that pairs with the two-dimensional Ladyzhenskaya inequality in the quadratic-moment argument.  Third, after fixing a pair of indices $i<j$ and conditioning the $N$-particle Gibbs measure on the pair center and all outside particle positions, the conditional density of the relative coordinate contains the vanishing point weight $|p(u)|^{\beta/N}$, where $p$ is the scalar complex polynomial of degree $2N-3$ defined in \eqref{eq:relative-law}.  The logarithm of this weight has point singularities, and its total exponent is $\beta(2N-3)/N<2\beta$.  Fourth, $\log|\Delta_N|$ is pluriharmonic on the collision complement; the quadratic confinement supplies the entire positive Levi form in the weighted $\bar\partial$ estimate \eqref{eq:dbar}.

The isotropic quadratic confinement is used in three further ways.  It yields an exact Gaussian center-of-mass factor (see \Cref{prop:com}), and it supplies both the Gaussian tails in \Cref{thm:weighted} and the common-phase symmetry used in the permutation-invariant argument through \Cref{lem:phase}.  For a general uniformly convex potential, the center and relative coordinates need not decouple, the one-particle conditionals are no longer Gaussian, and common-phase invariance is absent unless the potential is radial.  Extending the proof would therefore require quantitative substitutes for these three inputs.

\subsection{Standard consequences and the modulated problem}

The standard semigroup and Herbst consequences of a logarithmic Sobolev inequality give exponential relative-entropy relaxation for the Coulomb overdamped Langevin dynamics and Gaussian concentration for Lipschitz observables, with constants uniform in $N$ and locally uniform in $\beta\in(0,\infty)$; see \cite[Thm.~5.2.1, Prop.~5.4.1, and~(5.4.2)]{BakryGentilLedoux2014}.  We do not record these routine consequences in detail here.

A more substantive problem concerns the modulated logarithmic Sobolev method of \cite{RosenzweigSerfaty2025}.  Generation of chaos along a mean-field evolution requires uniform finite-particle inequalities not only for the single quadratic equilibrium treated here but also for a class of external potentials or prescribed reference measures generated along the flow.  \cref{thm:main} does not transfer automatically to such modulated ensembles.  A natural next problem is to identify a stable class of confining potentials $V$ for which the corresponding full-labeled Gibbs measures generated by singular Coulomb/Riesz interactions satisfy $N$-uniform logarithmic Sobolev inequalities with constants controlled by quantitative properties of $V$.  Such a theorem would provide the missing ingredient for extending the modulated logarithmic Sobolev program to Coulomb/Riesz dynamics with singular forces beyond the one-dimensional results of \cite{RosenzweigSerfaty2025}.

\subsection{The Hilbert--Schmidt Coulomb/Riesz program}\label{sec:hilbert-schmidt-program}

The present theorem is one central nonperturbative case in a broader program for canonical Coulomb/Riesz ensembles at the diffusive temperature scale, but its planar whole-space proof should not be interpreted as a template for the entire program.  Throughout this subsection, write
\begin{equation}\label{eq:outlook-riesz-potential}
 \g_{\s}(x):=
 \begin{cases}
  \dfrac1{\s}|x|^{-\s},&\s\ne0,\\
  -\log|x|,&\s=0,
 \end{cases}
 \qquad \s<\d,
\end{equation}
for the whole-space logarithmic/Riesz potentials.  By an abuse of notation, we use the same symbol $\g_{\s}$ for the corresponding mean-zero spectral potential on the flat torus $\T^\d\simeq [0,1)^{\d}$, with its Fourier normalization fixed.  The forthcoming work discussed below is joint with Delgadino and Gvalani \cite{RosenzweigHilbertSchmidtLSIForthcoming}.

For the periodic spectral logarithmic/Riesz ensembles, we obtain full-labeled logarithmic Sobolev inequalities in the range
\begin{equation}
 \d\ge2,\qquad 0\le\s\le\d-2,\qquad \s<\d/2,
\end{equation}
at every finite inverse temperature.  The constants are uniform in $N$ and in $\beta\in[0,B]$ for each finite $B$.  The conclusion holds for every measurable one-particle reference density $q$ satisfying fixed bounds $0<q_-\le q\le q_+<\infty$; no derivative bound is required, and the ratio $q_+/q_-$ need not be close to one.  Prescribed-background modulated Gibbs measures with the same positive density bounds are included.  The Poincar\'e inequality holds throughout $0\le\s\le\d-2$, without the additional Hilbert--Schmidt restriction.  In particular, the periodic two- and three-dimensional Coulomb cases are included in the all-finite-temperature logarithmic Sobolev range.

The periodic argument separates the spectral-gap and entropy estimates.  A compact-space estimate gives uniform conditional Poincar\'e inequalities after the positions of all but one particle have been fixed: the fixed particles contribute positive charges of total mass $\beta(N-1)/N\le\beta$ to the potential seen by the remaining particle.  The estimate is uniform in their number and positions.  To pass to the full Poincar\'e inequality, we use a Markov process that replaces one particle at a time, proposing its new position according to the reference density $q$.  After adding a constant to the spectral kernel, write $G=\g_{\s}+c\ge0$.  A replacement of $x_i$ by $z$ occurs at rate
\begin{equation}
 \exp\left[-\frac\beta N\sum_{j\ne i}G(z-x_j)\right]q(z)\,\D z.
\end{equation}
The Fourier coefficients of $G$ are nonnegative, so convolution by $G$ is a nonnegative operator on $L^2(\T^\d)$.  This positivity is used to prove a uniform spectral gap for the replacement process for large $N$; a fixed-particle argument treats the remaining values of $N$.  The conditional Poincar\'e inequalities compare its Dirichlet form with the gradient form and give the full-labeled Poincar\'e inequality.  After centering at the thermal equilibrium, conditional entropy estimates and the repulsive partition bound from our joint work \cite[Theorem~2.2(1) and Remark~2.1]{DelgadinoGvalaniRosenzweigForthcoming} give a defective logarithmic Sobolev inequality with a defect independent of $N$.  Rothaus tightening removes this defect.  These arguments use the spectral kernel and compactness; they do not require a dimensional extension of the planar weighted-Gaussian and complex-analytic proof used here.

The whole-space program includes the planar logarithmic case $(\d,\s)=(2,0)$ under quadratic confinement, for which \Cref{thm:main} gives the stronger logarithmic Sobolev inequality at every fixed $\beta>0$.  For positive Riesz exponents, the forthcoming work gives a uniform Poincar\'e inequality under quadratic confinement for every $\d\ge3$ and $0<\s\le\d-2$, at every finite inverse temperature.  In Gaussian coordinates the interaction coefficient is $\lambda/N$ in the unordered pair sum, and the rate is uniform for $\lambda\in[0,B]$.  For the unscaled density proportional to
\begin{equation}
 \exp\left[-\beta\left(\frac12\sum_i|x_i|^2+
       \frac1N\sum_{i<j}\g_{\s}(x_i-x_j)\right)\right],
\end{equation}
the change of variables $z_i=\sqrt\beta\,x_i$ gives $\lambda=\beta^{1+\s/2}$; hence its Poincar\'e rate is bounded below by $\beta\kappa_B$ whenever $0<\beta^{1+\s/2}\le B$, for some $\kappa_B>0$ independent of $N$.  For these positive Riesz exponents, the Gaussian-reference Poincar\'e result extends to references proportional to $b e^{-V}$, where $V\in C^2$, $\kappa I\le D^2V\le KI$ for fixed $0<\kappa\le K<\infty$, and $b$ is measurable with fixed positive lower and finite upper bounds.

The Lean source for the periodic logarithmic Sobolev and Poincar\'e inequalities and the whole-space Riesz Poincar\'e inequalities discussed above is available in \cite{DelgadinoGvalaniRosenzweigSingularLSILean2026}.

A particle-number-uniform LSI for positive Riesz exponents in the whole space remains the subject of ongoing work.

For $\s<0$, the program has a different analytic character because the pair potentials extend continuously across collisions.  The same forthcoming work \cite{RosenzweigHilbertSchmidtLSIForthcoming} concerns full-labeled logarithmic Sobolev inequalities for the periodic spectral Riesz ensembles for every $\s<0$ and every finite inverse temperature, uniformly in $N$ and locally uniformly on compact inverse-temperature intervals, against any one-site reference whose closed gradient form satisfies a logarithmic Sobolev inequality.  In the confined whole-space setting, the announced extensions concern the inhomogeneous Bessel family for every $\s<0$ when the one-particle confined measures have a common logarithmic Sobolev constant on the inverse-temperature interval, and the homogeneous Riesz family for $-2<\s<0$ under the confinement assumptions of that work.  At $\s=-2$, interaction and quadratic confinement compete at the same order, while quadratic confinement does not normalize the homogeneous model for $\s<-2$.  The program also includes prescribed-reference extensions and quantitative lower bounds for the logarithmic Sobolev constants.

None of these extensions follows formally from the present planar theorem; each requires arguments adapted to its geometry and interaction.

\newcommand{\etalchar}[1]{$^{#1}$}
\providecommand{\bysame}{\leavevmode\hbox to3em{\hrulefill}\thinspace}
\providecommand{\MR}{\relax\ifhmode\unskip\space\fi MR }
\providecommand{\MRhref}[2]{%
  \href{http://www.ams.org/mathscinet-getitem?mr=#1}{#2}
}
\providecommand{\href}[2]{#2}

\makeatletter\enddoc@text\let\enddoc@text\relax\makeatother
\clearpage
\setcounter{section}{0}
\setcounter{equation}{0}
\setcounter{theorem}{0}
\renewcommand{\theHsection}{supp.\arabic{section}}
\renewcommand{\theHequation}{\theHsection.\arabic{equation}}
\renewcommand{\theHtheorem}{\theHsection.\arabic{theorem}}
\setlist[enumerate]{topsep=8pt plus 2pt minus 4pt,itemsep=4pt plus 2pt minus 1pt,parsep=4pt plus 2pt minus 1pt}
\markboth{MATTHEW ROSENZWEIG}{SUPPLEMENT ON REVISIONS AND PARTITION NORMALIZATION}
\pdfbookmark[0]{Supplement}{supplement.start}
\begin{center}
{\large\scshape Supplement to Uniform Logarithmic Sobolev Inequalities for the 2D Coulomb Gas at the Diffusive Temperature Scale\par}
\medskip
{\scshape Matthew Rosenzweig\par}
\medskip
{\small 15 October 2026\par}
\end{center}
\begin{quotation}\small
\noindent\textsc{Abstract.} We describe the mathematical revisions to the first arXiv version and their relation to the Lean~4 formalization. We also give the repulsive partition estimate in the normalization used in the paper. The two main theorem statements are unchanged.
\end{quotation}

\section{Revisions and formalization}\label{sec:revisions}

The comparison is with \href{https://arxiv.org/abs/2608.24863v1}{arXiv:2608.24863v1}, submitted on 25 August 2026. The main theorems keep their original parameter ranges and closed form domains. Where autoformalization followed a different proof route, the manuscript proofs were revised to follow that argument. We summarize the substantive changes below; all references prefixed by ``the paper'' refer to the accompanying manuscript.

\begin{enumerate}
\item \textit{Domains and center-of-mass factorization.}
In the paper's Section~\ref{sec:model-domain}, Proposition~\ref{prop:capacity} and Lemma~\ref{lem:constants-domain} give the local collision cutoffs and a global core sequence $q_n\to1$ in the form norm. Thus, $1-q_n\to0$ in that norm; each $1-q_n$ equals one near the collision set. Proposition~\ref{prop:com} defines the relative closed gradient form. Its proof constructs the lift $\widetilde v(c,r)=v(r)$ by outer cutoffs, then uses the averaged slice error to obtain relative core approximants. Equations~\eqref{eq:relative-lift-cutoff} and~\eqref{eq:relative-slice-error} give these two approximation steps. Tensorization and entropy lower semicontinuity pass the product inequalities to the closed domain; the same cutoff construction makes the converse cylinder tests admissible.

\item \textit{Local weighted geometry.}
The paper's Lemma~\ref{lem:strong-Ainf} uses the continuous finite-point weight, gives the cone angles and curvature masses, and proves completeness before applying Bonk--Lang. The raw local inequalities have positive ball radius, dilation $11$, and real locally Lipschitz test functions. Corollary~\ref{cor:local-gauss} writes out the comparison of normalized weighted averages, including the change of mean. The distributional curvature identity and the cone curvature used by the metric theorem are stated separately.

\item \textit{Super-Poincar\'e and logarithmic Sobolev inequalities.}
The paper's Proposition~\ref{prop:super-pi} gives the interpolation exponent and displays the inner and outer cutoff errors before absorption. The subsequent smooth dyadic truncations satisfy the chain-rule estimate~\eqref{eq:smooth-truncation-energy}; the transition function contributes its fixed derivative bound $K^2$. Equations~\eqref{eq:dyadic-spi-layer} and~\eqref{eq:dyadic-entropy-layer} give the scale choice and summable error, leading to~\eqref{eq:weighted-defective-lsi}. Form approximation and entropy lower semicontinuity extend this inequality to the closed domain. The bounded-function Rothaus inequality~\eqref{eq:rothaus-centering} is then applied to core functions centered in that domain. The final closed-domain limit is taken after tightening. The same order of operations is used in the paper's Section~\ref{sec:tightening}.

\item \textit{Collision poles and permutation-invariant functions.}
The paper's Lemma~\ref{lem:pole-removal} removes a collision pole at a point where only one coordinate pair coincides using shrinking weighted mass: the function $yf(y)$ extends across the normal coordinate $y=0$, and parity then removes the remaining possible pole. This includes the exponent endpoint two. A local Cauchy integral gives joint holomorphy, followed by extension across multiple collisions. Proposition~\ref{prop:symmetric-pi} explicitly assumes the closed-domain hypothesis and first proves the estimate on the collision-free core. If $Q$ is the holomorphic projection and $f_0=f-\E f$, then $f_0-Qf_0=f-Qf$; this permits the compactly supported $\bar\partial$ estimate before symmetrized core approximation.

\item \textit{Fixed-particle compactness and temperature limits.}
The paper's Lemma~\ref{lem:fixed-N} proves compactness from translation control and moment tightness of the ground-state-transformed core functions. Core approximation places each bounded set in the closed form domain arbitrarily close to a bounded core set. The ground-state isometry then transfers the resulting compactness to the closed domain. The compact-temperature argument uses the exact ground-state energy comparison, bounded drift integral, and closedness at the limiting temperature. The displayed energy limit applies to the entire right-hand side.

\item \textit{Label-dependent variance.}
The paper's Proposition~\ref{prop:swap} bounds a conditional swap difference by four times the variance of the original relative-coordinate slice. The conditional law is even, has total point exponent less than $2\beta$, and is covered by the weighted Poincar\'e inequality. Integrating the conditional estimate and using~\eqref{eq:relative-gradient} gives the full-labeled bound.

\item \textit{Thermal equilibrium.}
The paper's Proposition~\ref{prop:equilibrium} specifies the finite-free-energy domain. Its convexity argument gives the integration-by-parts identity and extends it by mass-preserving mollification, normalized truncation, self-energy convergence, and lower semicontinuity. For positivity, a bounded probability density supported on a bounded part of the zero set has finite absolute self- and cross-energies. The exact entropy formula for its disjoint mixture with the minimizer yields the contradiction. The Euler--Lagrange equation is first obtained on the positive set; radial Newton formulas then give the regularity used in the proof.

\item \textit{Modulation, partition function, and entropy.}
The paper's Section~\ref{sec:equilibrium} records the absolute integrability needed for the signed spatial-diagonal interaction. Its distinct-label expansion is made on collision-free configurations; the resulting density identities hold product-almost everywhere. Section~\ref{sec:partition-supplement} below supplies the repulsive partition bound, with the half-interaction factor and the one-body correction. The lower bound is $\exp(tc/2)$. In the paper's conditional-entropy proof, the exponential estimate controls the negative part of $B_{N,\beta}=A_{N,\beta}+N^{-1}\sum_iU_\beta(x_i)$; Gaussian-polynomial tails and boundedness of $r^{\beta/N}(-\log r)_+$ control its positive part. Thus the full integrability premise holds before the Gibbs measure is substituted into the entropy inequality.
\end{enumerate}

The two main theorems have been formalized in Lean~4, together with the mathematical results invoked by their formal proofs, whether proved in the paper or taken from the literature. External results are formalized in the forms used in these proofs. The formalization uses the actual weighted Gaussian and labeled Coulomb measures, their closed gradient-form domains, and the paper's entropy conventions. It constructs the thermal reference used in the partition and entropy estimates. The source is available in \cite{S-RosenzweigPlanarCoulombLean2026}; the repository records the verification provenance for the selected statements and their proof dependencies.

Anthropic's Claude and OpenAI's ChatGPT/Codex were used for autoformalization under the author's close supervision. AI-assisted comparisons of the manuscript with the formal source were also used to check statements, intermediate steps, hypotheses, and limiting arguments. The present revisions include a separate review of the changed arguments. These correspondence reviews concern the written mathematical arguments; the Lean verification concerns formal proof terms. The author is responsible for the manuscript and its mathematical claims.

\clearpage
\section{The repulsive partition estimate in the canonical normalization}\label{sec:partition-supplement}

The conditional-entropy argument in the paper uses the exponential moment of a modulated interaction, rather than an uncentered pair sum. We give the partition calculation in that normalization, so that its half-interaction factor, one-body correction, and uniform dependence on the density bound are explicit. This is the estimate applied in the paper's partition proposition and subsequent defect-moment bound.

Fix $\beta>0$ and let $\mu=\mu_\beta=\rho_\beta\D x$ be the thermal reference constructed in the paper.  It satisfies $0\le\rho_\beta\le1/\pi$ and has finite second moment.  For $X_N=(x_1,\ldots,x_N)$, write $\Delta=\{(x,x):x\in\C\}$ and $\g(z)=-\log|z|$,
\begin{equation}
 U=\g*\mu,\qquad c=\int U\D\mu,\qquad
 \mu_N=\frac1N\sum_{i=1}^N\delta_{x_i},\qquad
 \eta_N=\sqrt N(\mu_N-\mu).
\end{equation}
The logarithmic row and product integrals are finite: bounded density controls the singularity, and the second moment controls the logarithmic growth at infinity.  Thus, on collision-free configurations,
\begin{align}
 A_N&:=\frac12\iint_{\C^2\setminus\Delta}\g(x-y)\,\D\eta_N(x)\D\eta_N(y)
       =N\mathsf F_N(X_N,\mu)\notag\\
 &=\frac1{2N}\sum_{i\ne j}\g(x_i-x_j)-\sum_iU(x_i)+\frac N2c.
 \label{eq:supp-A}
\end{align}
The configurations excluded here have $\mu^{\otimes N}$-measure zero.  All random variables below are understood up to that null set.

\begin{proposition}\label{prop:supp-partition}
For each $T>0$ there is $C_T<\infty$, independent of $\beta>0$ and $N\ge2$, such that
\begin{equation}
 K_{N,t}(\mu):=\E_{\mu^{\otimes N}}[e^{-tA_N}]\le\max\{1,C_T\},
 \qquad 0<t\le T.
\end{equation}
Moreover, $K_{N,t}(\mu)\ge e^{tc/2}>0$.
\end{proposition}

The upper bound is the conclusion of \cite[Theorem~2.5]{S-DelgadinoGvalani2025}, with the centered fluctuation $\eta_N=\sqrt N(\mu_N-\mu)$ and the half-interaction convention in \eqref{eq:supp-A}. Appendix~A of that reference verifies the logarithmic-kernel assumptions. We explain two corrections to the formulas in arXiv:2506.22083v1 and their effect on the upper-bound argument.

First, the definition of the fluctuation on pp.~6--7 should read
\begin{equation}\label{eq:supp-centering}
 \eta_N=\frac1{\sqrt N}\left(\sum_{i=1}^N\delta_{x_i}-N\mu\right)
       =\sqrt N(\mu_N-\mu).
\end{equation}
The printed definition omits the factor $N$ before $\mu$; the subsequent expansion in the proof of Lemma~3.1 uses the centered coefficients in \eqref{eq:supp-A}.

Second, the regularization error has a one-body remainder. We use the double-circle regularization from the formal proof: if $\sigma_\varepsilon$ is uniform probability on the circle of radius $\varepsilon$, set $\g_\varepsilon=\g*\sigma_\varepsilon*\sigma_\varepsilon$ and define $A_{N,\varepsilon}$ by replacing $\g$ by $\g_\varepsilon$ in every term of \eqref{eq:supp-A}. The cited paper uses heat regularization; the algebraic identity below applies to either choice. Set
\begin{align}
 &H_\varepsilon=\g-\g_\varepsilon,\qquad
 a_\varepsilon=H_\varepsilon*\mu,\qquad
 b_\varepsilon=\int a_\varepsilon\D\mu,\label{eq:supp-error-kernel}\\
 &G_\varepsilon(x,y)=H_\varepsilon(x-y)-a_\varepsilon(x)-a_\varepsilon(y)+b_\varepsilon.
\end{align}
Then $G_\varepsilon$ is symmetric with zero $\mu$-row means, and expansion gives
\begin{equation}\label{eq:supp-error}
 A_N-A_{N,\varepsilon}
 =\frac1{2N}\sum_{i\ne j}G_\varepsilon(x_i,x_j)
   -\frac1N\sum_i a_\varepsilon(x_i)+\frac12 b_\varepsilon.
\end{equation}
This is the identity needed in place of the regularization-error formula in \cite[equation~(3.8), p.~12]{S-DelgadinoGvalani2025}. In the half-interaction convention, that formula omits both the factor $1/2$ in front of the pair sum and the last two terms in \eqref{eq:supp-error}.

For our regularization, write $h_\varepsilon(z)=\log(\varepsilon/|z|)_+$. The circle-potential formula gives $H_\varepsilon=h_\varepsilon+\sigma_\varepsilon*h_\varepsilon$ off zero. The density bound $\rho_\beta\le1/\pi$, integration on the disk of radius $\varepsilon$, and Jensen's inequality give, for $0<\varepsilon\le1$ and integers $p\ge1$,
\begin{equation}\label{eq:supp-row}
 \sup_x\int H_\varepsilon(x-y)^p\D\mu(y)
 \le 2^{p+1}p!\varepsilon^2,\qquad
 \sup_x\int|G_\varepsilon(x,y)|^p\D\mu(y)\le C_p\varepsilon^2.
\end{equation}
In particular, $0\le a_\varepsilon,b_\varepsilon\le4\varepsilon^2$, so
\begin{equation}\label{eq:supp-remainder}
 \left|-\frac1N\sum_i a_\varepsilon(x_i)+\frac12b_\varepsilon\right|
 \le6\varepsilon^2.
\end{equation}
Thus the extra contribution to the $p$th error moment is at most $C_p\varepsilon^{2p}$, uniformly in $N$. It is absorbed in the regularization estimates of the cited correlation-moment argument; the factor $1/2$ only changes its constants. The remaining deterministic bounds are
\begin{equation}\label{eq:supp-regularized-lower}
 A_{N,\varepsilon}\ge\tfrac12\log\varepsilon,\qquad
 A_N\ge-\tfrac14\log N-\tfrac12.
\end{equation}
They follow from positivity of the logarithmic energy of zero-mass measures, respectively applied to $\eta_N*\sigma_\varepsilon$ and to $\sqrt N(\mu_N*\sigma_\varepsilon-\mu)$ with $\varepsilon=N^{-1/2}$. These measures have finite absolute logarithmic energy, by the bounded density, finite second moment, and integrability of the logarithmic singularity on circles. The correlation inequality of \cite[Lemma~3.2]{S-DelgadinoGvalani2025} and the exponential-tail argument in its proof of Theorem~2.5 therefore give the upper bound. For the finitely many $N$ below the moment order used there, the second bound in \eqref{eq:supp-regularized-lower} suffices. All constants depend on $T$ and the common density bound, so they are uniform in the thermal parameter $\beta$.

The lower bound follows separately from independence and Jensen's inequality:
\begin{align}
 \E_{\mu^{\otimes N}}[A_N]&=\frac{N-1}{2}c-Nc+\frac N2c=-\frac c2,\label{eq:supp-mean}\\
 K_{N,t}(\mu)&\ge\exp\bigl(-t\E_{\mu^{\otimes N}}[A_N]\bigr)=e^{tc/2}>0.
\end{align}
The printed lower bound $1$ in \cite[Theorem~2.5]{S-DelgadinoGvalani2025} is not used. Indeed, dilation of the reference measure by $R>0$ adds $(\log R)/2$ to $A_N$ and multiplies $K_{N,t}$ by $R^{-t/2}$. The lower bound in the present whole-space normalization is consequently $e^{tc/2}$.

\providecommand{\bysame}{\leavevmode\hbox to3em{\hrulefill}\thinspace}
\providecommand{\MR}{\relax\ifhmode\unskip\space\fi MR }
\providecommand{\MRhref}[2]{%
  \href{http://www.ams.org/mathscinet-getitem?mr=#1}{#2}
}
\providecommand{\href}[2]{#2}


\begin{thebibliography}{CGWW09}

\bibitem[AB19]{AkemannByun2019}
Gernot Akemann and Sung-Soo Byun, \emph{The high temperature crossover for
  general {2D Coulomb} gases}, J. Statist. Phys. \textbf{175} (2019), no.~6,
  1043--1065.

\bibitem[ABC{\etalchar{+}}00]{aneInegalitesSobolevLogarithmiques2000}
C{\'e}cile An{\'e}, S{\'e}bastien Blach{\`e}re, Djalil Chafa{\"i}, Pierre
  Foug{\`e}res, Ivan Gentil, Florent Malrieu, Cyril Roberto, and Gr{\'e}gory
  Scheffer, \emph{Sur les in{\'e}galit{\'e}s de {Sobolev} logarithmiques},
  Panoramas et Synth{\`e}ses, vol.~10, Soci{\'e}t{\'e} Math{\'e}matique de
  France, Paris, 2000.

\bibitem[AIM09]{AstalaIwaniecMartin2009}
Kari Astala, Tadeusz Iwaniec, and Gaven Martin, \emph{Elliptic partial
  differential equations and quasiconformal mappings in the plane}, Princeton
  Mathematical Series, vol.~48, Princeton University Press, Princeton, NJ,
  2009.

\bibitem[BBD25]{BauerschmidtBodineauDagallier2025}
Roland Bauerschmidt, Thierry Bodineau, and Benoit Dagallier, \emph{A criterion
  on the free energy for {log-Sobolev} inequalities in mean-field particle
  systems}, 2025, arXiv:2503.24372.

\bibitem[BCF18]{BCF2018}
Fran{\c{c}}ois Bolley, Djalil Chafa{\"i}, and Joaqu{\'i}n Fontbona,
  \emph{Dynamics of a planar {Coulomb} gas}, Ann. Appl. Probab. \textbf{28}
  (2018), no.~5, 3152--3183.

\bibitem[B{\'E}85]{BakryEmery1985}
Dominique Bakry and Michel {\'E}mery, \emph{Diffusions hypercontractives},
  S{\'e}minaire de Probabilit{\'e}s XIX, 1983/84, Lecture Notes in Mathematics,
  vol. 1123, Springer, Berlin, 1985, pp.~177--206.

\bibitem[BGL14]{BakryGentilLedoux2014}
Dominique Bakry, Ivan Gentil, and Michel Ledoux, \emph{Analysis and geometry of
  {Markov} diffusion operators}, Grundlehren der mathematischen Wissenschaften,
  vol. 348, Springer, 2014.

\bibitem[Bj{\"o}01]{Bjorn2001}
Jana Bj{\"o}rn, \emph{{Poincar{\'e}} inequalities for powers and products of
  admissible weights}, Ann. Acad. Sci. Fenn. Math. \textbf{26} (2001), no.~1,
  175--188.

\bibitem[BL03]{BonkLang2003}
Mario Bonk and Urs Lang, \emph{Bi-{Lipschitz} parameterization of surfaces},
  Math. Ann. \textbf{327} (2003), no.~1, 135--169.

\bibitem[BM26]{BeckerMenegaki2026}
Simon Becker and Angeliki Menegaki, \emph{Spectral gaps for mean-field
  {Coulomb} gases}, 2026, arXiv:2608.21215v1.

\bibitem[BR03]{BartheRoberto2003}
Franck Barthe and Cyril Roberto, \emph{{Sobolev} inequalities for probability
  measures on the real line}, Studia Math. \textbf{159} (2003), no.~3,
  481--497.

\bibitem[Bre11]{Brezis2011}
Haim Brezis, \emph{Functional analysis, {Sobolev} spaces and partial
  differential equations}, Universitext, Springer, New York, 2011.

\bibitem[CG17]{CattiauxGuillin2017}
Patrick Cattiaux and Arnaud Guillin, \emph{Hitting times, functional
  inequalities, {Lyapunov} conditions and uniform ergodicity}, J. Funct. Anal.
  \textbf{272} (2017), no.~6, 2361--2391.

\bibitem[CGWW09]{CGWW2009}
Patrick Cattiaux, Arnaud Guillin, Feng-Yu Wang, and Liming Wu, \emph{Lyapunov
  conditions for super {Poincar{\'e}} inequalities}, J. Funct. Anal.
  \textbf{256} (2009), no.~6, 1821--1841.

\bibitem[Cha26]{Chafai2026Ginibre}
Djalil Chafa{\"i}, \emph{An optimal {Poincar{\'e}} inequality for the complex
  {Ginibre} log-gas}, 2026, arXiv:2608.19358.

\bibitem[CL20]{ChafaiLehec2020}
Djalil Chafa{\"i} and Joseph Lehec, \emph{On {Poincar{\'e}} and logarithmic
  {Sobolev} inequalities for a class of singular {Gibbs} measures}, Geometric
  Aspects of Functional Analysis, Lecture Notes in Mathematics, vol. 2256,
  Springer, 2020, pp.~219--246.

\bibitem[CNZ24]{ChewiNitandaZhang2024}
Sinho Chewi, Atsushi Nitanda, and Matthew~S. Zhang, \emph{Uniform-in-{$N$}
  {log-Sobolev} inequality for the mean-field {Langevin} dynamics with convex
  energy}, 2024, arXiv:2409.10440.

\bibitem[Dem12]{DemaillyCADG}
Jean-Pierre Demailly, \emph{Complex analytic and differential geometry},
  Institut Fourier, Universit{\'e} de Grenoble I, 2012, Online version, 21 June
  2012.

\bibitem[DG25]{DelgadinoGvalani2025}
Matias~G. Delgadino and Rishabh~S. Gvalani, \emph{Sharp mean-field estimates
  for the repulsive log gas in any dimension}, 2025, arXiv:2506.22083v1.

\bibitem[DGPS23]{DelgadinoGvalaniPavliotisSmith2023}
Matias~G. Delgadino, Rishabh~S. Gvalani, Grigorios~A. Pavliotis, and Scott~A.
  Smith, \emph{Phase transitions, logarithmic {Sobolev} inequalities, and
  uniform-in-time propagation of chaos for weakly interacting diffusions},
  Comm. Math. Phys. \textbf{401} (2023), no.~1, 275--323.

\bibitem[DGR26a]{DelgadinoGvalaniRosenzweigSingularLSILean2026}
Matias Delgadino, Rishabh Gvalani, and Matthew Rosenzweig, \emph{{Lean}
  formalization of periodic logarithmic {Sobolev} and {Poincar{\'e}}
  inequalities and whole-space {Riesz} {Poincar{\'e}} inequalities},
  \url{https://github.com/matthewrosenzweigwork-max/DGR-Singular-LSI-Lean-Public},
  2026, {Lean} source repository, commit \texttt{933ab9a}.

\bibitem[DGR26b]{DelgadinoGvalaniRosenzweigForthcoming}
Matias~G. Delgadino, Rishabh Gvalani, and Matthew Rosenzweig, \emph{Sharp
  mean-field estimates for diffusive log/{Riesz} gases in the
  {Hilbert--Schmidt} regime}, 2026, arXiv:2609.02743v1.

\bibitem[DGR26c]{RosenzweigHilbertSchmidtLSIForthcoming}
\bysame, \emph{Uniform logarithmic {Sobolev} inequalities for diffusive
  log/{Riesz} gases in the {Hilbert--Schmidt} regime}, Work in preparation,
  2026.

\bibitem[DS81]{DiaconisShahshahani1981}
Persi Diaconis and Mehrdad Shahshahani, \emph{Generating a random permutation
  with random transpositions}, Z. Wahrsch. Verw. Gebiete \textbf{57} (1981),
  no.~2, 159--179.

\bibitem[DS90]{DavidSemmes1990}
Guy David and Stephen Semmes, \emph{Strong {$A_\infty$} weights, {Sobolev}
  inequalities and quasiconformal mappings}, Analysis and Partial Differential
  Equations, Lecture Notes in Pure and Applied Mathematics, vol. 122, Dekker,
  1990, pp.~101--111.

\bibitem[FMRT01]{FoiasEtAl2001}
C.~Foias, O.~Manley, R.~Rosa, and R.~Temam, \emph{{Navier--Stokes} equations
  and turbulence}, Encyclopedia of Mathematics and its Applications, vol.~83,
  Cambridge University Press, Cambridge, 2001.

\bibitem[FOT11]{FukushimaOshimaTakeda2011}
Masatoshi Fukushima, Yoichi Oshima, and Masayoshi Takeda, \emph{Dirichlet forms
  and symmetric {Markov} processes}, 2nd revised and extended ed., De Gruyter
  Studies in Mathematics, vol.~19, Walter de Gruyter, Berlin, 2011.

\bibitem[GLBM23]{GuillinLeBrisMonmarche2023}
Arnaud Guillin, Pierre Le~Bris, and Pierre Monmarch{\'e}, \emph{On systems of
  particles in singular repulsive interaction in dimension one: log and {Riesz}
  gas}, J. {\'E}c. polytech. Math. \textbf{10} (2023), 867--916.

\bibitem[GLWZ22]{GuillinLiuWuZhang2022}
Arnaud Guillin, Wei Liu, Liming Wu, and Chaoen Zhang, \emph{Uniform
  {Poincar{\'e}} and logarithmic {Sobolev} inequalities for mean field particle
  systems}, Ann. Appl. Probab. \textbf{32} (2022), no.~3, 1590--1614.

\bibitem[Gra14]{Grafakos2014}
Loukas Grafakos, \emph{Classical {Fourier} analysis}, 3rd ed., Graduate Texts
  in Mathematics, vol. 249, Springer, New York, 2014.

\bibitem[Gro75]{Gross1975}
Leonard Gross, \emph{Logarithmic {Sobolev} inequalities}, Amer. J. Math.
  \textbf{97} (1975), no.~4, 1061--1083.

\bibitem[GZ03]{GuionnetZegarlinski2003}
Alice Guionnet and Boguslaw Zegarlinski, \emph{Lectures on logarithmic
  {Sobolev} inequalities}, S{\'e}minaire de Probabilit{\'e}s XXXVI, Lecture
  Notes in Mathematics, vol. 1801, Springer, Berlin, 2003, pp.~1--134.

\bibitem[GZ19]{GarciaZelada2019}
David Garc{\'i}a-Zelada, \emph{A large deviation principle for empirical
  measures on {Polish} spaces: Application to singular {Gibbs} measures on
  manifolds}, Ann. Inst. Henri Poincar{\'e} Probab. Stat. \textbf{55} (2019),
  no.~3, 1377--1401.

\bibitem[Han78]{Han1978}
Te~Sun Han, \emph{Nonnegative entropy measures of multivariate symmetric
  correlations}, Information and Control \textbf{36} (1978), no.~2, 133--156.

\bibitem[H{\"o}r65]{Hormander1965}
Lars H{\"o}rmander, \emph{{$L^2$} estimates and existence theorems for the
  {$\bar\partial$} operator}, Acta Math. \textbf{113} (1965), 89--152.

\bibitem[HS87]{HolleyStroock1987}
Richard Holley and Daniel Stroock, \emph{Logarithmic {Sobolev} inequalities and
  stochastic {Ising} models}, J. Statist. Phys. \textbf{46} (1987), no.~5--6,
  1159--1194.

\bibitem[KZC{\etalchar{+}}24]{KookZhangChewiErdogduLi2024}
Yunbum Kook, Matthew~S. Zhang, Sinho Chewi, Murat~A. Erdogdu, and Mufan~(Bill)
  Li, \emph{Sampling from the mean-field stationary distribution}, Proceedings
  of Thirty Seventh Conference on Learning Theory, Proceedings of Machine
  Learning Research, vol. 247, PMLR, 2024, pp.~3099--3136.

\bibitem[Lam21]{Lambert2021}
Gaultier Lambert, \emph{Poisson statistics for {Gibbs} measures at high
  temperature}, Ann. Inst. Henri Poincar{\'e} Probab. Stat. \textbf{57} (2021),
  no.~1, 326--350.

\bibitem[LM20]{LuMattingly2020}
Yulong Lu and Jonathan~C. Mattingly, \emph{Geometric ergodicity of {Langevin}
  dynamics with {Coulomb} interactions}, Nonlinearity \textbf{33} (2020),
  no.~2, 675--699.

\bibitem[Mon24]{Monmarche2024}
Pierre Monmarch{\'e}, \emph{Uniform {log-Sobolev} inequalities for mean field
  particles beyond flat-convexity}, 2024, arXiv:2409.17901.

\bibitem[MT10]{MadimanTetali2010}
Mokshay Madiman and Prasad Tetali, \emph{Information inequalities for joint
  distributions, with interpretations and applications}, IEEE Trans. Inform.
  Theory \textbf{56} (2010), no.~6, 2699--2713.

\bibitem[OR07]{OttoReznikoff2007}
Felix Otto and Maria~G. Reznikoff, \emph{A new criterion for the logarithmic
  {Sobolev} inequality and two applications}, J. Funct. Anal. \textbf{243}
  (2007), no.~1, 121--157.

\bibitem[Ros26a]{RosenzweigPlanarCoulombLean2026}
Matthew Rosenzweig, \emph{{Lean} formalization of uniform logarithmic {Sobolev}
  inequalities for the {2D Coulomb} gas at the diffusive temperature scale},
  \url{https://github.com/matthewrosenzweigwork-max/LSI-2D-Diffusive-Coulomb-Gas-Lean-Public},
  2026, {Lean} source repository, commit \texttt{c1faa35}.

\bibitem[Ros26b]{RosenzweigFiniteGinibreNotes2026}
\bysame, \emph{Permutation-invariant functional inequalities for finite
  {Ginibre} {I}: The exact {Poincar{\'e}} rate and the centered {LSI}
  frontier}, Unpublished note, 12 August 2026;
  \href{https://matthewrosenzweigwork-max.github.io/assets/notes/permutation-invariant-functional-inequalities-finite-ginibre-20260812.pdf}{available
  online}, 2026.

\bibitem[Ros26c]{RosenzweigInfiniteGinibreNotes2026}
\bysame, \emph{A spectral gap for the unlabeled infinite {Ginibre} interacting
  {Brownian} motion}, Unpublished note, 11 August 2026;
  \href{https://matthewrosenzweigwork-max.github.io/assets/notes/spectral-gap-unlabeled-infinite-ginibre-20260811.pdf}{available
  online}, 2026.

\bibitem[Ros26d]{RosenzweigBlogPartI}
\bysame, \emph{Uniform logarithmic {Sobolev} inequalities for the {2D Coulomb}
  gas at the diffusive temperature scale, {Part~I}}, Research exposition on the
  author's website, August 2026,
  \url{https://matthewrosenzweigwork-max.github.io/posts/uniform-logarithmic-sobolev-2d-coulomb-gas-part-i/}.

\bibitem[Rot85]{Rothaus1985}
Oscar~S. Rothaus, \emph{Analytic inequalities, isoperimetric inequalities and
  logarithmic {Sobolev} inequalities}, J. Funct. Anal. \textbf{64} (1985),
  no.~2, 296--313.

\bibitem[RS25]{RosenzweigSerfaty2025}
Matthew Rosenzweig and Sylvia Serfaty, \emph{Modulated logarithmic {Sobolev}
  inequalities and generation of chaos}, Ann. Fac. Sci. Toulouse Math. (6)
  \textbf{34} (2025), no.~1, 107--134.

\bibitem[RW01]{RocknerWang2001}
Michael R{\"o}ckner and Feng-Yu Wang, \emph{Weak {Poincar{\'e}} inequalities
  and {$L^2$}-convergence rates of {Markov} semigroups}, J. Funct. Anal.
  \textbf{185} (2001), no.~2, 564--603.

\bibitem[Sem93]{Semmes1993}
Stephen Semmes, \emph{Bi-{Lipschitz} mappings and strong {$A_\infty$} weights},
  Ann. Acad. Sci. Fenn. Math. \textbf{18} (1993), 211--248.

\bibitem[Ser26]{Serfaty2024Lectures}
Sylvia Serfaty, \emph{Lectures on {Coulomb} and {Riesz} gases}, American
  Mathematical Society Colloquium Publications, vol.~70, American Mathematical
  Society, Providence, RI, 2026.

\bibitem[Suz26]{Suzuki2026Ginibre}
Kohei Suzuki, \emph{Spectral gap for unlabelled {Ginibre} interacting
  {Brownian} motion}, 2026, arXiv:2608.17437.

\bibitem[Wan05]{Wang2005}
Feng-Yu Wang, \emph{Functional inequalities, {Markov} semigroups and spectral
  theory}, Science Press, Beijing--New York, 2005.

\bibitem[Wan24]{Wang2024}
Songbo Wang, \emph{Uniform {log-Sobolev} inequalities for mean field particles
  with flat-convex energy}, 2024, arXiv:2408.03283.

\end{thebibliography}

\begin{thebibliography}{Ros26}

\bibitem[DG25]{S-DelgadinoGvalani2025}
Matias~G. Delgadino and Rishabh~S. Gvalani, \emph{Sharp mean-field estimates
  for the repulsive log gas in any dimension}, 2025, arXiv:2506.22083v1.

\bibitem[Ros26]{S-RosenzweigPlanarCoulombLean2026}
Matthew Rosenzweig, \emph{{Lean} formalization of uniform logarithmic {Sobolev}
  inequalities for the {2D Coulomb} gas at the diffusive temperature scale},
  \url{https://github.com/matthewrosenzweigwork-max/LSI-2D-Diffusive-Coulomb-Gas-Lean-Public},
  2026, {Lean} source repository, commit \texttt{c1faa35}.

\end{thebibliography}
\end{document}